\documentclass[11pt]{amsart}
\usepackage{graphicx,amssymb,amsfonts,epsfig,amsthm,a4,amsmath,url,amscd,mathrsfs,xspace,hyperref}
\usepackage{graphicx}
\usepackage[francais,english]{babel}
\usepackage{inputenc}
\usepackage[T1]{fontenc}
\usepackage[centering]{geometry}
\usepackage{csquotes}
\usepackage{color}
\usepackage{enumerate}
\usepackage{enumitem}

  \newcommand{\R}{\ensuremath{\mathbb{R}}}%
  \newcommand{\Z}{\ensuremath{\mathbb{Z}}}%
  \newcommand{\N}{\ensuremath{\mathbb{N}}}%
                \renewcommand{\P}{\ensuremath{\mathcal{P}}}%

                \renewcommand{\Pi}{\ensuremath{\mathcal{P}_\perp}}%

        \renewcommand{\H}{\ensuremath{\mathcal{H}}}%

  \newcommand{\Sym}{\ensuremath{\operatorname{Sym}}}%
    \newcommand{\sym}{\ensuremath{\operatorname{Sym}}}%
      \newcommand{\Alt}{\ensuremath{\operatorname{Alt}}}%
    \newcommand{\alt}{\ensuremath{\operatorname{Alt}}}%
    \newcommand{\Att}{\ensuremath{\operatorname{Att}}}%
    \newcommand{\Per}{\ensuremath{\operatorname{Per}}}%
    \newcommand{\Rep}{\ensuremath{\operatorname{Rep}}}%

  \newcommand{\supp}{\ensuremath{\operatorname{Supp}}}%
	\newcommand{\fix}{\ensuremath{\operatorname{Fix}}}%
    \newcommand{\acts}{\ensuremath{\curvearrowright}}%
  \newcommand{\sub}{\ensuremath{\operatorname{Sub}}}%
  \newcommand{\homeo}{\ensuremath{\operatorname{Homeo}}}%
    \newcommand{\rist}{\ensuremath{\operatorname{R}}}%

			\newcommand{\Tsf}{\mathsf{T}}
		\newcommand{\Hsf}{\mathsf{H}}
				\newcommand{\PL}{\mathsf{PL}}

\theoremstyle{definition}
  \newtheorem{defin}{Definition}[section]

\theoremstyle{plain}
  \newtheorem{thm}[defin]{Theorem}
  \newtheorem{main thm}{Theorem}
  \newtheorem{prop}[defin]{Proposition}
    \newtheorem{prop-def}[defin]{Proposition-Definition}
    \newtheorem{thmintro}{Theorem}
    \newtheorem{corintro}[thmintro]{Corollary}

  \newtheorem{cor}[defin]{Corollary}
   \newtheorem{lem}[defin]{Lemma}

\theoremstyle{remark}
  \newtheorem{remark}[defin]{Remark}

\begin{document}

  \date{November 19, 2021}	
\title{Confined subgroups and high transitivity}

\author{Adrien Le Boudec}
\address{CNRS, UMPA - ENS Lyon, 46 all\'ee d'Italie, 69364 Lyon, France}
\email{adrien.le-boudec@ens-lyon.fr}

\author{Nicol\'as Matte Bon}
\address{
	CNRS,
	Institut Camille Jordan (ICJ, UMR CNRS 5208),
	Universit\'e de Lyon,
	43 blvd.\ du 11 novembre 1918,	69622 Villeurbanne,	France}

\email{mattebon@math.univ-lyon1.fr}

\thanks{ALB supported by ANR-19-CE40-0008-01 AODynG}

\maketitle

\begin{abstract}
An action of a group $G$ is highly transitive if $G$ acts transitively on $k$-tuples of distinct points for all $k \geq 1$. Many examples of groups with a rich geometric or dynamical action admit highly transitive actions. We prove that if a group $G$ admits a highly transitive action such that $G$ does not contain the subgroup of finitary alternating permutations, and if $H$ is a confined subgroup of $G$, then the action of $H$ remains highly transitive, possibly after discarding finitely many points.

This result provides a tool to rule out the existence of highly transitive actions, and to classify highly transitive actions of a given group. We give concrete illustrations of these applications in the realm of groups of dynamical origin. In particular we obtain the first non-trivial classification of highly transitive actions of a finitely generated group. 
\end{abstract}


\section{Introduction}

Given a group $G$, we let $\sub(G)$ be the space of subgroups of $G$, endowed with the induced topology from the set $2^G$ of all subsets of $G$. The group $G$ acts on $\sub(G)$ by conjugation. A subgroup $H$ of $G$ is called a \textbf{confined subgroup} if the closure of the $G$-orbit of $H$ in $\sub(G)$ does not contain the trivial subgroup. Equivalently, $H$ is a confined subgroup if there exists a finite subset $P$ of non-trivial elements of $G$ such that $g H g^{-1} \cap P$ is not empty for every $g \in G$. A subgroup $H$ of $G$ is a  uniformly recurrent subgroup (URS)  if the orbit closure of $H$ is minimal \cite{GW-urs}, i.e.\ does not contain any proper closed $G$-invariant subspace. These two notions are closely related, as every non-trivial subgroup $H$ that is a URS is confined; and conversely if $H$ is confined, then the orbit closure of $H$ contains a non-trivial URS. 

An early appearance of the notion of confined subgroups can be traced back to the study of ideals in group algebras of locally finite groups, notably in \cite{Se-Za-alt,Zaless94,Hart-Zal-97}. In particular the terminology was introduced in \cite{Hart-Zal-97}. Confined subgroups and URSs are intimately connected to stabilizers of group actions on compact spaces. In this context they have been implicitly considered in the literature under different names, see \cite{Grig-Nek-Sus,Grig-survey2011}.  Recently these notions turned out to be useful tools to understand reduced group $C^*$-algebras, thanks to a connection with topological boundary actions \cite{Kal-Ken,Ken-URS,BKKO}. The systematic study of confined subgroups and URSs of various countable groups in \cite{LBMB-sub-dyn,LB-ame-urs-latt,MB-graph-germs,LBMB-comm-lemma} has led to several applications and rigidity results.

 In this article we establish a connection between confined subgroups and highly transitive actions. Throughout the paper  $\Omega$ will always be an infinite set, and we denote by $\Sym(\Omega)$ the group of all permutations of $\Omega$. An action of a group $G$ on $\Omega$ is $k$-transitive if $G$ acts transitively on the set of $k$-tuples of distinct elements of $\Omega$. An action is \textbf{highly transitive} if it is $k$-transitive for all $k$. Equivalently, the action of $G$ on $\Omega$ is highly transitive if the induced group homomorphism from $G$ to $\Sym(\Omega)$ has a dense image, where $\Sym(\Omega)$ is endowed with the pointwise convergence topology.

An example of a highly transitive action is for instance given by the action of the group $\Sym_f(\Omega)$ of finitary permutations of $\Omega$, or even its subgroup $\Alt_f(\Omega)$ of alternating permutations. Here a permutation is termed finitary if it moves only finitely many points. Recall that $\Alt_f(\Omega)$ and $\Sym_f(\Omega)$ are normal subgroups of $\Sym(\Omega)$, and every non-trivial normal subgroup of $\Sym(\Omega)$ contains $\Alt_f(\Omega)$ \cite{Onofri,Sch-Ul}.  For a group $G$ acting faithfully and highly transitively on $\Omega$, one easily verifies that the image of $G$ in $\Sym(\Omega)$ contains a non-trivial finitary permutation if and only if it contains $\Alt_f(\Omega)$. In this situation we will say that $G$ is \textbf{partially finitary}. This can be reformulated as an intrinsic property of the group, see \S  \ref{subsec-finitary}.  These are somehow considered as trivial examples, and we are mostly interested in highly transitive actions of groups that are not partially finitary. 

The first goal of this article is to prove the following result, which asserts that high transitivity is essentially inherited by confined subgroups. Whenever $H$ is a subgroup of $\Sym(\Omega)$, we denote by $\Omega_{H,f}$ the union  of the finite $H$-orbits in $\Omega$.

\begin{thmintro} \label{thm-intro-main}
	Suppose that a group $G$ admits a faithful and highly transitive action on $\Omega$ and that $G$ is not partially finitary. If $H$ is a confined subgroup of $G$, then  $\Omega_{H,f}$ is finite, and the action of $H$ on $\Omega \setminus \Omega_{H,f}$ is highly transitive.
\end{thmintro}

\begin{remark}
In the setting of the theorem, the set $\Omega_{H,f}$ is not empty in general, for instance because the stabilizer in $G$ of a point or of a finite subset of $\Omega$ may very well be a confined subgroup of $G$. See e.g. Remark \ref{remark-fullgp} or Section \ref{sec-classification}.
\end{remark}

\begin{remark}
In the case where $G$ is partially finitary, a similar conclusion holds, up to passing to a finite $H$-invariant partition of $\Omega \setminus \Omega_{H,f}$; see Theorem \ref{thm-conf-ht-precise}. 
\end{remark}

Recently there has been a growing interest for a better understanding of which (say countable) groups admit faithful highly transitive actions. Actually the question is already rich for primitive actions, and was investigated in depth by Gelander and Glasner \cite{Gel-Glas-prim}, partly motivated by earlier results of Margulis and Soifer \cite{Margulis-Soifer}. See \cite{GGS} for a recent survey. A series of consecutive works shows that various groups admitting a rich geometric or dynamical action also admit a highly transitive action. This is for instance the case for free groups \cite{McDonough-free,Dixon-free,Olshanskii-free,Eis-Glas,LeMaitre-FG}, surface groups \cite{Kitroser-surface}, hyperbolic groups \cite{Chaynikov}, outer automorphism groups of free groups \cite{Garion-Glasner}, acylindrically hyperbolic groups  \cite{Hull-Osin}, unbounded Zariski dense subgroups of $\mathrm{SL}_2$ over a local field \cite{GGS}, and many groups acting on trees \cite{FLMMS}. Hence Theorem \ref{thm-intro-main}  applies in all these situations.

On the other hand the study of the existence of highly transitive actions within a given class of groups still suffers from the lack of obstructions that prevent a group from having such actions. In other words, given a group $G$, we know very few restrictions on $G$ that are imposed by the existence of a faithful highly transitive action of $G$. See \S \ref{elementary-obs} for a brief discussion around this. The following consequence of Theorem \ref{thm-intro-main} provides a new criterion to rule out the existence of highly transitive actions.

\begin{corintro} \label{cor-intro-criteria}
Let $G$ be a group that is not partially finitary, and let $H$ be a confined subgroup of $G$. If $H$ does not admit any faithful highly transitive action, then neither does $G$.
\end{corintro}

We illustrate this criterion by providing an application to a class of groups of dynamical origin, namely subgroups  of the groups of piecewise linear homeomorphisms of suspension flows introduced in \cite{MB-T-flows}. We briefly recall this construction and refer to  \S \ref{s-PL-flows} for detailed definitions.
Let $X$ be a Stone space (i.e. a compact totally disconnected space), and $\varphi$ be a homeomorphism of $X$. The pair $(X, \varphi)$ will be called a \textbf{Stone system}. Let $Y^\varphi$ be the suspension space (or mapping torus) of $(X, \varphi)$, and $\Phi\colon \R\times Y^\varphi\to Y^\varphi$ be the associated suspension flow. The space $Y^\varphi$ is locally homeomorphic to $X\times \R$, and its path-connected components are exactly the $\Phi$-orbits. 
Following \cite{MB-T-flows}, we let $\PL(\varphi)$ be the group of all homeomorphisms of $Y^\varphi$ which preserve every $\Phi$-orbit and act as piecewise linear homeomorphism of $\R$ on each of them, and whose displacement along every $\Phi$-orbit is locally constant in the transverse direction. 

Recall that an action of a group $G$ on a compact space $Y$ by homeomorphisms is \textbf{topologically free} if  the set $\fix(g)$ of fixed points of $g$ has empty interior for every non-trivial element $g$ of $G$. Conversely the action is \textbf{topologically nowhere free} if every point $y$ in $Y$ is contained in the interior of $\fix(g)$ for some non-trivial element $g$ of $G$. This condition implies that  the stabiliser $G_y$ of every $y\in Y$ is a confined subgroup of $G$.  Using Corollary \ref{cor-intro-criteria}, we show the following result.

\begin{thmintro} \label{t-intro-PL}
	Let $(X, \varphi)$ be a minimal Stone system. Let $G$  be a subgroup of $\PL(\varphi)$ whose action on $Y^\varphi$ is topologically nowhere free.  Then $G$ does not admit any faithful highly transitive action. 
\end{thmintro}

Examples of groups to which the criterion applies are the groups $\Tsf(\varphi)\le \PL(\varphi)$ introduced and studied in \cite{MB-T-flows}, which are analogues in $\PL(\varphi)$ of Thompson's group $T$ acting on the circle. It is shown in  \cite{MB-T-flows} that the group $\Tsf(\varphi)$ is simple when the system $(X, \varphi)$ is minimal,  and finitely generated when $(X, \varphi)$ is conjugate to a subshift over a finite alphabet. In \S \ref{subsec-variant} we describe a related family of examples to which Theorem \ref{t-intro-PL} also applies, which retain similar properties, and for which we are additionally able to check that none of the previously known obstructions to high transitivity applies.

\bigskip

Beyond the problem of deciding which groups admit highly transitive actions, it is natural to try to classify all highly transitive actions of a given group. In the case of primitive actions, similar questions have been considered by Gelander--Glasner in \cite{Gel-Glas-prim,GelGlas-ONanScot} for countable linear groups. It follows from an old result of Jordan-Wielandt that the group $\alt_f(\Omega)$ admits a unique highly transitive action up to conjugacy, and more generally the same is true for every partially finitary group (see \S \ref{subsec-finitary}). Apart from this observation, there was so far no example of a finitely generated highly transitive group whose highly transitive actions are completely understood. Theorem \ref{thm-intro-main}  provides a tool towards such classification results. We illustrate this mechanism by considering the family of Higman--Thompson's groups $V_d$ \cite{Hig-fp}, whose definition is recalled in \S \ref{subsec-prelim-V}. It is well-known that for the natural action of $V_d$ on the Cantor set, the action on each orbit is highly transitive. Using Theorem \ref{thm-intro-main}, we prove the following:

\begin{thmintro} \label{thmintro-Vd}
Every faithful highly transitive action of the Higman--Thompson group $V_d$ is conjugate to its natural action on an orbit in the Cantor set.
\end{thmintro}

Here we focus on the family of groups $V_d$ instead of trying to reach the most optimal level of generality, but the argument is likely to be adaptable to other groups acting on compact spaces with suitable dynamical properties (see Section  \ref{sec-classification}, and in particular Remark \ref{rmq-general-Vd}).

\bigskip

To end this introduction, we point out that the analogue of Theorem \ref{thm-intro-main} is also true in the measured setting of  invariant random subgroups (IRSs), i.e. $G$-invariant Borel probability measures on $\sub(G)$ \cite{AGV}. That is, whenever $G$ is not partially finitary, high transitivity is inherited by all IRSs of $G$;  see Proposition  \ref{prop-IRS-ht}.  This fact is a rather direct application of the classical de Finetti theorem from probability theory, whose importance to the study of IRSs was recognised by Vershik \cite{Ver-IRS}. We point out that we are not aware of any application of this statement, for instance to  classify or to rule out the existence of highly transitive actions of certain groups. In the groups considered in Theorem \ref{t-intro-PL} and Theorem \ref{thmintro-Vd}, the conjugacy class closure in $\sub(G)$ of the confined subgroups that are involved does not carry any $G$-invariant probability measure. In other words, we deal with URSs that are not IRSs. Hence the above result for IRSs does not provide any help here, and the appropriate tool is indeed Theorem \ref{thm-intro-main}.

\subsection*{Outline}

The proof of Theorem \ref{thm-intro-main} is given in Section  \ref{sec-conf-sym}. The crucial point of the proof is Theorem \ref{thm-confi-SYM}, which provides a description of the confined subgroups of the symmetric group $\Sym(\Omega)$. The reason why this result is relevant regarding Theorem \ref{thm-intro-main} comes from the observation that confined subgroups behave well with respect to a dense embedding of a group $G$ into a topological group $L$: the image closure in $L$ of every confined subgroup of $G$ is a confined subgroup of $L$ (see Lemma  \ref{lem-conf-dense-embed}). Since a faithful highly transitive action of $G$ on $\Omega$ is the same as a dense embedding into $\Sym(\Omega)$, every confined subgroup of $G$ thus gives rise to a confined subgroup of $\Sym(\Omega)$. With Theorem \ref{thm-confi-SYM} and this observation in hand, the proof of Theorem \ref{thm-intro-main} follows easily. In the course of the proof of Theorem \ref{thm-confi-SYM}, we make use of the \textit{commutator lemma} for confined subgroups established in \cite{LBMB-comm-lemma}. Roughly speaking, this result says that when a group admits an action with non-trivial rigid stabilizers, then under certain assumptions every confined subgroup must contain a large part of some rigid stabilizer. We refer to Theorem \ref{thm-conf-abstract}  for a precise statement and for the terminology. As a remark, we note that it seems reasonable to believe that our mechanism of proof here with the help of Theorem \ref{thm-conf-abstract} might be applicable to study dense embeddings into other large topological groups.

Section \ref{sec-appl} is dedicated to Theorem \ref{t-intro-PL}. An ingredient of the proof is an auxiliary result about the subquotients of the group $\PL(I)$ of piecewise linear homeomorphisms of an interval, which  implies in particular that every finite subquotient of the group $\PL(I)$ must be solvable; see Proposition \ref{prop-metab-subq}. That result is an improvement of the Brin-Squier result \cite{Brin-Squier} asserting that $\PL(I)$ does not contain any non-abelian free subgroup, and provides an immediate obstruction to high transitivity for subgroups of $\PL(I)$ (see \S \ref{elementary-obs}). This conclusion does not generalize to the group  $\PL(\varphi)$ (which contains free subgroups). However, it does generalize to the stabilizers of the natural action of $\PL(\varphi)$ on $Y^\varphi$  (Lemma \ref{l-first-return}). Since the latter are confined subgroups, we deduce Theorem \ref{t-intro-PL} from Corollary~\ref{cor-intro-criteria}.

Section  \ref{sec-classification} contains the proof of Theorem \ref{thmintro-Vd}, and is independent of Section \ref{sec-appl}. The proof relies on an appropriate combination of Theorem \ref{thm-intro-main} together with an independent result about proximality of the action of certain subgroups of $V_d$ on the Cantor set (Proposition \ref{p-V-prox}). 

\subsection*{Acknowledgements}

We thank the referee for a careful reading of the article and for useful comments.

\section{Preliminaries} \label{sec-prelim}

\subsection{Notation for group actions}\label{s-notations-stab}

Suppose that a group $G$ acts on a set $Y$. We denote by $\fix(g)\subset Y$ the set of fixed points of an element $g\in G$, and define its support  as the set $\supp(g)=Y\setminus \fix(g)$. We use the notation $G_y$ for the stabiliser of a point $y\in Y$, and $G_{y_1,\ldots, y_n}$ for the pointwise fixator of an $n$-tuple of points.   If $Y$ has a topology, the \textbf{germ-stabilizer} of  $y$ is the subgroup $G^0_y$ consisting of elements that fix pointwise some neighbourhood of $y$. Similarly $G^0_{y_1,\cdots, y_n}$ is the subgroup of elements that fix pointwise a neighbourhood of each of the points $y_1,\ldots, y_n$.  

\subsection{Confined subgroups}

Let $H,G$ be two subgroups of a group $L$. We say that $H$ is \textbf{confined} by $G$ if there is a finite subset $P$ of non-trivial elements of $L$ such that $g H g^{-1} \cap P$ is not empty for every $g \in G$. Such a subset $P$  is called a \textbf{confining subset} for the pair $(H,G)$.  A \textbf{confined subgroup} of a group is a subgroup which is confined by the whole ambient group. Note that if $H$ contains a subgroup confined by a group $G$, then $H$ is also confined by $G$. 

Suppose that a group $G$ acts on a compact space $Y$. We say that the action is \textbf{topologically nowhere free} if $G^0_y\neq\{1\}$ for every $y\in Y$. The notions of confined subgroups and topologically nowhere free actions are related through the following lemma.

\begin{lem} \label{lem-conf-stab}
Suppose that the action of $G$ on the compact space $Y$ is topologically nowhere free. Then for every $y\in Y$, the germ stabilizer $G^0_y$ (and hence the stabilizer $G_y$) is a confined subgroup of $G$.
\end{lem}

\begin{proof}
For every $y$ one can find a non-trivial element of $G$ acting trivially on a neighbourhood of $y$. Hence by compactness one can find a finite subset $P$ of non-trivial elements such that the interiors of the set of fixed points of elements of $P$ cover $Y$, and it easily follows that $P$  is confining for the pair $(G^0_y,G)$ for every $y\in Y$.
\end{proof}

\subsection{Partially finitary groups} \label{subsec-finitary}

Recall from the introduction that we call a group $G$ partially finitary if $G$ admits an embedding into the symmetric group $\Sym(\Omega)$ over an infinite set such that the image of $G$ contains $\Alt_f(\Omega)$. 

\begin{lem}
	A group $G$ is partially finitary  if and only if there exists an infinite set $\Omega$ and an injective group homomorphism $\varphi: \Alt_f(\Omega) \to G$ such that the image of $\varphi$ is normal and has trivial centralizer in $G$.
\end{lem}

\begin{proof}
	The direct implication is clear, and the converse follows from the fact that the conjugation action of $G$ on the image of $\varphi$ provides an injective homomorphism from $G$ into the automorphism group of $\Alt_f(\Omega)$. Since the latter is $\Sym(\Omega)$ \cite[Th. 8.2.A]{Dixon-Mortimer}, the claim follows. 
\end{proof}	
	
The following is a classical result from the theory of permutation groups due to Jordan-Wielandt  \cite[Th. 3.3.D]{Dixon-Mortimer}. Recall that an action of a group $G$ on a set $\Omega$ is \textbf{primitive} if it preserves no partition of $\Omega$ other than the one block partition and the partition into singletons. For transitive actions, this is equivalent to saying that point stabilizers are maximal subgroups. For instance any $2$-transitive action is primitive. 

\begin{thm} \label{thm-Jor-Wie}
If the action of a subgroup $G \leq \Sym(\Omega)$ on the infinite set $\Omega$ is primitive and if $G$ contains a non-trivial finitary permutation, then $G$ contains $\Alt_f(\Omega)$.
\end{thm}

In particular this result has the following easy consequence:

\begin{prop}
If a group $G$ is partially finitary, then $G$ admits a unique faithful 2-transitive action, which is highly transitive. 
\end{prop}

\begin{proof}
Suppose that $G$ admits an embedding into $\Sym(\Omega)$ such that the image of $G$ contains $A = \Alt_f(\Omega)$.  We claim that  actually the action of $G$ on $\Omega$ is its only $2$-transitive action. Indeed suppose $H$ is a point stabilizer for a faithful $2$-transitive action of $G$. Since $A$ is normal in $G$, the partition of $G/H$ into $A$-orbits forms a $G$-invariant partition, which is not the partition into points as $A$ must act faithfully. Since a $2$-transitive action is primitive, it follows that this partition must be trivial, and thus the action of $A$ on $G/H$ is transitive. If $A \cap H$ is trivial, then $G = A \rtimes H$ and the action of $H$ on $G/H$ coincides with the conjugation action of $H$ on $A$. This is clearly impossible since non-trivial elements of $A$ are not all conjugate. Therefore $H$ must intersect $A$ non-trivially, thus contains elements with finite support on $\Omega$. However $H$ cannot contain the whole group $A$, as we already noted that $A$ acts transitively on $G/H$, and thus cannot fix $H$. From this we deduce that the action of $H$ on $\Omega$ is not primitive by the Jordan-Wielandt theorem. Hence by maximality of $H$ in $G$ we see that $H$ is the stabilizer of a non-trivial partition $\P$ of $\Omega$, and the 2-transitive action we started with is the action on the $G$-orbit $\mathcal{O}_\P$ of $\P$.  Note that the action of $G$ on pairs of distinct elements $(\P_1, \P_2)$ of  $\mathcal{O}_\P$ must preserve the number of blocks of $\P_1$ which do not belong to $\P_2$. If $\P$ has at least 3 blocks, then by considering images of $\P$ under elements of $A$ it is easy to find two pairs for which this number is finite and different, contradicting 2-transitivity. We deduce that $\P$ must be of the form $\left\lbrace \Omega_1,\Omega_2\right\rbrace$ with $\Omega_1, \Omega_2$ non-empty. Equivalently, $H$ is just the stabilizer of $\Omega_1$ in $G$, and the action of $G$ on $G/H$ is the action of $G$ on the $G$-orbit of $\Omega_1$. But a very similar reasoning shows that this  action can be $2$-transitive only when $\Omega_1$ or the complement of $\Omega_1$ is a singleton. So $H$ is a point stabilizer for the $G$-action on $\Omega$, and the claim is proven. \qedhere
\end{proof}

\subsection{A brief discussion about high transitivity} \label{elementary-obs}

In this paragraph we discuss some basic facts about highly transitive actions, and collect some  observations that allow to exclude the existence of highly transitive actions of a given group. As far as we know, at the moment these observations provide the only known obstructions to the existence of highly transitive actions of a countable group.

Let $G$ be a group admitting a faithful highly transitive action on $\Omega$. The first preliminary observation is that the action of every non-trivial normal subgroup of $G$ remains highly transitive. This follows for instance from the classical result that every non-trivial normal subgroup of $\Sym(\Omega)$ contains the group $\Alt_f(\Omega)$ \cite{Onofri,Sch-Ul}, since highly transitive actions correspond to homomorphisms to $\Sym(\Omega)$ with dense image. Therefore in particular every non-trivial normal subgroup of $G$ has trivial centralizer in $G$. 

If $\Sigma$ is a finite subset of $\Omega$ of cardinality $k$, then the stabilizer of $\Sigma$ in $G$ admits a homomorphism to $\Sym(\Sigma)$, that is surjective by the definition of high transitivity. Hence the group $G$ admits all finite symmetric groups (hence all finite groups) as subquotients. Recall that $Q$ is a \textbf{subquotient} of a group $G$ if $G$ admits a subgroup $H$ such that $Q$ is isomorphic to a quotient of $H$. This implies for instance that $G$ cannot be of bounded exponent, or more generally that $G$ cannot be a torsion group whose element orders are divisible by only finitely many primes. This also implies that $G$ cannot satisfy a non-trivial \textbf{identity} (or law), i.e.\ there does not exist a non-trivial element $w$ in the free group  $\mathbb{F}_n$ of rank $n$ such that every homomorphism from $\mathbb{F}_n$ to $G$ is trivial on $w$.

This generalizes to the more general notion of mixed identities. If $w$ is an element of the free product $G \ast \mathbb{F}_n$, we say that $G$ satisfies the \textbf{mixed-identity} $w=1$ if every homomorphism from $G \ast \mathbb{F}_n$ to $G$ that is the identity on $G$ is trivial on $w$. We say that $G$ satisfies a (non-trivial) mixed-identity if one can find $n \geq 1$ and a non-trivial element $w \in G \ast \mathbb{F}_n$ such that $G$ satisfies the mixed-identity $w=1$. By \cite[Theorem 5.9]{Hull-Osin}, if a group $G$ admits a faithful highly transitive action and $G$ is not partially finitary, then $G$ does not satisfy a mixed-identity. {Basic examples of groups that satisfy a non-trivial mixed identity are non-trivial direct products, and more generally any group that admits two non-trivial commuting normal subgroups.} Note that such groups may very well not satisfy an identity, for example because they might contain free subgroups.  Slightly less obvious  examples are groups admitting a fixed point free action by homeomorphisms on the real line such that some non-trivial elements have a relatively compact support, e.g.\ Thompson's group $F$ \cite{id-F}. Similar but slightly more elaborated phenomena happen for certain groups acting on the circle or acting on trees \cite[Prop. 3.7-4.7]{LBMB-trans-degree}.

\section{Confined subgroups of infinite symmetric groups and Theorem \ref{thm-intro-main}} \label{sec-conf-sym}

\subsection{Confined subgroups of infinite symmetric groups } \label{subsec}

Whenever $H$ is a subgroup of $\Sym(\Omega)$, we denote by $\Omega_{H,f}$ the union  of the finite $H$-orbits in $\Omega$, and by $\Omega_{H,\infty}$ the union of the infinite $H$-orbits. Given a partition $\mathcal{Q}=(\Omega_i)_{i\in I}$ of $\Omega$, the subsets $\Omega_i$ are called the \textbf{blocks} of $\mathcal{Q}$. We say that $\mathcal{Q}$ is  a \textbf{thick partition} if every block of $\mathcal{Q}$ has cardinality at least two.

\begin{lem} \label{lem-conf-ht}
	Let $H \leq  \Sym(\Omega)$ be a subgroup that is confined by a subgroup $G \leq  \Sym(\Omega)$ that is highly transitive on $\Omega$, and let $P$ be a confining subset for $(H,G)$ of cardinality $r=r_1+r_2$, where $r_1$ and $r_2$ are respectively the number of elements of $P$ with finite support and with infinite support. Then the following hold:
	\begin{enumerate}[label=(\roman*)]
		\item \label{item-conf-sym-fix} $H$ has at most $r-1$ fixed points in $\Omega$.
		\item  \label{item-conf-sym-orb} $\Omega_{H,f}$ is finite.
		\item \label{item-conf-sym-partition} Every $H$-invariant thick partition $\mathcal{Q} = (\Omega_i)_{i \in I}$ of $\Omega$ is finite; and: \begin{enumerate}
			\item if $r_1=0$ then $\mathcal{Q}$ has exactly one infinite block $\Omega_{i_0}$, and $|\Omega \setminus \Omega_{i_0}| \leq r-1$;
			\item if $r_1 \geq 1$ then $\mathcal{Q}$ has at most $r_1$ infinite blocks.
		\end{enumerate}
	\end{enumerate}
\end{lem}

\begin{proof}
	\ref{item-conf-sym-fix}. Suppose for a contradiction that $H$ fixes $r$ distinct fixed points $x_1,\ldots,x_r$. We choose a set $\Sigma \subseteq \Omega$ of cardinality at most $r$ such that every element of $P$ moves a point in $\Sigma$. Since the action of $G$ is highly transitive, one can find $g \in G$ such that $g(\Sigma) \subseteq \left\lbrace x_1,\ldots,x_r \right\rbrace$. Then every element of $gPg^{-1}$ moves at least one $x_i$, and hence $gPg^{-1} \cap H$ is empty, a contradiction.

	We show the statement \ref{item-conf-sym-partition}. Note that \ref{item-conf-sym-orb} will follow in view of \ref{item-conf-sym-fix}. Assume that $\mathcal{Q} = (\Omega_i)_{i \in I}$ is a thick $H$-invariant partition of $\Omega$ . We decompose $P = P_1 \cup P_2$, where $P_1$ is the set of elements of $P$ having finite support. Let $S\subset \Omega$ be the union of the supports of elements of $P_1$. We write $ P_2 = \left\{\sigma_{r_1+1}, \ldots, \sigma_r\right\}$. Note that for every finite set $F\subset \Omega$ and every $\sigma_i\in P_2$, there exists $x\in \Omega$ such that $x \neq \sigma_i(x)$ and $x, \sigma_i(x)\notin  F$ (just choose $x\in \supp(\sigma_i)\setminus (F\cup \sigma_i^{-1}(F))$, which is non-empty since the support $\supp(\sigma_i)$ is infinite). Using this observation repeatedly,  we can choose points $x_i$ and $x_i'$ such that \[ x_{r_1+1}, x_{r_1+1}', \ldots, x_r, x_r', \sigma_{r_1+1}(x_{r_1+1}), \sigma_{r_1+1}(x_{r_1+1}') \ldots, \sigma_r(x_r), \sigma_r(x_r') \] are all distinct and all outside $S$.
	
	Now we write \[ T = S \cup \left\{x_{r_1+1}, x_{r_1+1}', \ldots, x_r, x_r', \sigma_{r_1+1}(x_{r_1+1}), \sigma_{r_1+1}(x_{r_1+1}') \ldots, \sigma_r(x_r), \sigma_r(x_r')\right\}. \] Assume for a contradiction that $\mathcal{Q}$ has infinitely many blocks. Since the group $G$ is highly transitive and all blocks have cardinality at least two, we can find $g \in G$ with the following properties:

	\begin{itemize}
		\item for any two distinct $s,s' \in S$, $g(s)$ and $g(s')$ belong to distinct blocks;
		\item for all $i$ there exist ${j(i)}, {k(i)},{l(i)}$ such that $\Omega_{k(i)},\Omega_{l(i)}$ are distinct and $g(x_i), g(x_i') \in \Omega_{j(i)}$, $g(\sigma_i(x_i)) \in \Omega_{j(i)}$ and $g(\sigma_i(x_i')) \in \Omega_{l(i)}$.
	\end{itemize}
	
	For such an element $g$, we claim that $gPg^{-1}$ cannot intersect $H$. Indeed, assume first that there is $h \in H \cap gP_1g^{-1}$. Then on the one hand the element $h$ must preserve the partition, and on the other hand the support of $h$ intersects every block of the partition along at most a singleton. This is of course impossible, so $gP_1g^{-1}$ does not intersect $H$. Now if $i$ is such that $g\sigma_ig^{-1}$ belongs to $H$, then the block $\Omega_{j(i)}$ must be sent to both $\Omega_{k(i)}$ and $\Omega_{l(i)}$, which is again impossible. So we have obtained a contradiction, as desired. So $\mathcal{Q}$ is finite.
	
	Assume now that $r_1=0$, i.e.\ every element of $P$ has infinite support. It follows from the previous paragraph that $\mathcal{Q}$ has at least one infinite block $\Omega_{i_0}$. We have to argue that the complement of $\Omega_{i_0}$ has cardinality at most $r-1$. Suppose this is not the case. Keeping the same notation as before, again by high transitivity of the group $G$ we can find $g \in G$ such that $g(x_i), g(x_i'), g(\sigma_i(x_i)) \in \Omega_{i_0}$ and $g(\sigma_i(x_i')) \notin \Omega_{i_0}$ for all $i$. This implies that none of the elements $g\sigma_ig^{-1}$ sends $\Omega_{i_0}$ onto a block of $\mathcal{Q}$. Hence $gPg^{-1}$ does not intersect $H$, and again we have a contradiction. 
	
	Finally we assume that $r_1 \geq 1$ and that $\mathcal{Q}$ has at least $r_1 + 1 \geq 2$ infinite blocks, say $\Omega_{j_1},\ldots, \Omega_{j_{r_1+1}}$. We choose a finite subset $Y=\{y_1,\ldots, y_k\}\subset S$ of cardinality $k\le r_1$ such that every element of $P_1$ moves a point in $Y$. We choose $g \in G$ such that $g(y_i)\in \Omega_{j_i}$  for all $i=1,\ldots, k$ and $g(S\setminus Y)\subset \Omega_{j_{k+1}}$.

	 Since in addition $r_1+1\ge 2$, we can also arrange $g$ to satisfy $g(x_i), g(x_i'), g(\sigma_i(x_i))\in \Omega_{j_1}$ and $g(\sigma_i(x_i'))\in \Omega_{j_2}$ for all $i=r_1+1, \ldots, r$. We claim that $H$ cannot intersect $gPg^{-1}$. Indeed assume by contradiction that $H$ contains $g\sigma g^{-1}$ with $\sigma\in P_1$. Let $i$ be such that $\sigma(y_i)\neq y_i$.  Then $g\sigma g^{-1}$ moves $g(y_i)\in \Omega_{j_i}$. However  the support of $g\sigma g$ intersects $\Omega_{j_1}$ only along $\{g(y_i)\}$, as  $g(S)\cap \Omega_{j_i}=\{g(y_i)\}$ by construction. Hence $g\sigma g^{-1}$ must necessarily preserve the block $\Omega_{j_i}$ and hence fix $g(y_i)$,  which is absurd. Finally all elements of  $g P_2g^{-1}$ must send one of the points $g(x_i)\in \Omega_{j_1}$ inside $\Omega_{j_1}$ and the corresponding point $g(x_i)'\in \Omega_{j_1}$ inside $\Omega_{j_2}$, so that none of them can preserve $\mathcal{Q}$. Thus $H$ does not intersect $g P_2g^{-1}$. 
 \end{proof}

In the sequel we borrow terminology and notation from \cite[\S 3]{LBMB-comm-lemma}.
	\begin{defin} \label{def-displace}
	Let $P$ be a finite set of non-trivial elements of $\Sym(\Omega)$, and let $\left\lbrace \Omega_\sigma \right\rbrace_{\sigma \in P} $  be a family of non-empty subsets of $\Omega$ indexed by $P$. We say that $\left\lbrace \Omega_\sigma \right\rbrace_{\sigma \in P} $ is a \textbf{displacement configuration} for $P$ if the following hold: 
	\begin{enumerate}[label=\roman*)]
		\item \label{item-disj-eq} for all $\sigma, \rho \in P$, either $\Omega_\sigma = \Omega_\rho$, or $\Omega_\sigma$ and $\Omega_\rho$ are disjoint;
		\item \label{item-one-move} for all $\sigma, \rho \in P$, either $\sigma$ fixes $\Omega_\rho$ pointwise, or $\sigma(\Omega_\rho)$ is disjoint from $\bigcup_{\alpha \in P}\Omega_\alpha$;
		\item \label{item-all-disj} for all $\sigma \in P$, $\sigma(\Omega_\sigma)$ is disjoint from $\bigcup_{\alpha \in P}\Omega_\alpha$ and also from $\bigcup_{\alpha \in P} \sigma^{-1}(\Omega_\alpha)$.
	\end{enumerate}
\end{defin}

Note that by the last condition, the subsets $\sigma^{-1}(\Omega_\sigma), \Omega_\sigma, \sigma(\Omega_\sigma)$ are disjoint for  all $\sigma \in P$. In particular this prevents $\sigma $ from being of order two. We will make use of the following easy lemma, which is a partial converse of the previous observation. 

\begin{lem} \label{lem-displconf-perm}
	Let $\sigma_1,\ldots,\sigma_q \in \Sym(\Omega)$ such that $\sigma_i^2$ has infinite support for all $i$. Then there exist infinite subsets $\Sigma_1,\ldots,\Sigma_q \subset \Omega$ that form a displacement configuration for $\left\{\sigma_1,\ldots,\sigma_q\right\}$.
\end{lem}

\begin{proof}
Lemma 4.1 in \cite{LBMB-comm-lemma} states that if $X$ is a Hausdorff space without isolated points, and if $\sigma_1,\ldots, \sigma_q$ are homeomorphisms of $X$ such that $\sigma_i^2$ is non-trivial, then there is a displacement configuration for $\{\sigma_1,\ldots, \sigma_q\}$ consisting of open subsets. The present statement can be deduced  from that one as follows. 	Consider the action of $\Sym(\Omega)$ on the Stone-\v{C}ech compactification $\beta \Omega$. Recall that there is a one-to-one correspondence between subsets of $\Omega$ and clopen subsets of $\beta \Omega$, given by $\Lambda \mapsto U_\Lambda$, where $U_\Lambda$ is the set of ultrafilters $\omega \in \beta \Omega$ such that $\omega(\Lambda) =1$. The fact that $\sigma_i^2$ moves infinitely many points in $\Omega$ is equivalent to saying that $\sigma_i^2$ does not act trivially on $\beta^* \Omega := \beta \Omega \setminus \Omega$. Hence we are in position to apply \cite[Lemma 4.1]{LBMB-comm-lemma} to $\left\{\sigma_1,\ldots,\sigma_q\right\}$. By this lemma, there exist open subsets $U_1, \ldots, U_q$ of $\beta^*  \Omega$ that form a displacement configuration for $\left\{\sigma_1,\ldots,\sigma_q\right\}$ in $\beta^*  \Omega$, and the existence of infinite subsets $\Sigma_1,\ldots,\Sigma_q$ that form a displacement configuration for $\left\{\sigma_1,\ldots,\sigma_q\right\}$ in $\Omega$ easily follows. 
\end{proof}

\begin{remark}
Alternatively, the lemma can be easily proven directly, without using the  Stone-\v{C}ech compactification, by adapting the argument of \cite[Lemma 4.1]{LBMB-comm-lemma}. 
\end{remark}

The following is Theorem 3.17 from \cite{LBMB-comm-lemma}. Whenever $\Sigma$ is a subset of $\Omega$ and $G$ is a group acting on $\Omega$, we denote by $\rist_G(\Sigma)$ the pointwise fixator of the complement of $\Sigma$, and call it the \textbf{rigid stabilizer} of $\Sigma$. In the sequel by the trivial conjugacy class we mean the singleton consisting of the trivial element of the group.

\begin{thm}[Commutator lemma for confined subgroups] \label{thm-conf-abstract}
	Let $H, G \le \Sym(\Omega)$ be two subgroups such that $H$ is confined by $G$, and let $P$ be a confining subset for $(H,G)$, and $r = |P|$. Assume that $\left\lbrace \Omega_\sigma \right\rbrace_{\sigma \in P} $ is a displacement configuration for $P$ such that for all $\sigma \in P$ the group $\rist_G(\Omega_\sigma)$ is non-trivial, and every non-trivial conjugacy class in $\rist_G(\Omega_\sigma)$ has cardinality $> r$.
	
	Then there exists $\rho \in P$ such that $H$ contains a non-trivial subgroup $N\le \rist_G(\Omega_\rho)$ whose normalizer in $\rist_G(\Omega_\rho)$ has index at most $r$. 
\end{thm}

We will also use  the  following simplified version of \cite[Lemma 3.2]{LBMB-comm-lemma}.
\begin{lem}\label{l-abelian-confined}
Let $G, H$ be subgroups of a group $L$ such that $H$ is confined by $G$. Suppose that every confining subset $P\subset L$ for $(H, G)$ contains an element of order 2. Then there exists an abelian subgroup $K\le L$ which is confined by $G$.

\end{lem}

The following result provides a description of the confined subgroups of $\Sym(\Omega)$. If we think of confined subgroups as generalizations of normal subgroups, it might be compared to the classical result that every non-trivial normal subgroup of $\Sym(\Omega)$ contains $\Alt_f(\Omega)$ \cite{Onofri,Sch-Ul}, or to the description of all normal subgroups of $\Sym(\Omega)$ \cite[Th.\ 8.1.A]{Dixon-Mortimer}. For a subset $\Sigma$ of $\Omega$, the rigid stabilizer of $\Sigma$ in  $\Alt_f(\Omega)$ is naturally isomorphic to $\Alt_f(\Sigma)$. By abuse of notation we will denote by $\Alt_f(\Sigma)$ this rigid stabilizer.

\begin{thm} \label{thm-confi-SYM}
	For a subgroup $H \leq \Sym(\Omega)$, the following are equivalent:
	\begin{enumerate}[label=\roman*)]
		\item \label{item-confbysym}  $H$ is a confined subgroup of $\Sym(\Omega)$;
		\item \label{item-confbyalt} $H$ is confined by $\Alt_f(\Omega)$;
		\item \label{item-nicesbgp}  $\Omega_{H,f}$ is finite, and there exists an $H$-invariant partition $\Omega_{H,\infty} = \Omega_1 \cup \ldots \cup  \Omega_k$ of $\Omega_{H,\infty} $ into infinite subsets $\Omega_1, \ldots, \Omega_k$ such that $H$ contains $\mathrm{Alt}_f(\Omega_1) \times \ldots \times \mathrm{Alt}_f(\Omega_k)$.
	\end{enumerate} 
\end{thm}

\begin{remark}
$\ref{item-confbysym} \implies \ref{item-confbyalt}$ is trivial and $\ref{item-nicesbgp}\Rightarrow \ref{item-confbysym}$ is an easy verification (see below).
The main content of the theorem is the implication $\ref{item-confbyalt} \implies \ref{item-nicesbgp}$. Theorem 1 in
\cite{Se-Za-alt} provides a characterization of the confined subgroups $H$ of $\Alt_f(\Omega)$, which turns out to be the same as the one obtained in Theorem  \ref{thm-confi-SYM}. In other words, that theorem gives the implication $\ref{item-confbyalt} \implies \ref{item-nicesbgp}$ \textit{within the group $\Alt_f(\Omega)$}. The argument there uses in an essential way the fact that all elements of $H$ have finite support, and hence does not apply to the situation of Theorem \ref{thm-confi-SYM}. 
\end{remark}

\begin{prop} \label{prop-confsym-precise}
	Let $H \leq \Sym(\Omega)$ be a subgroup that is confined by $\Alt_f(\Omega)$. Let $P$ be a confining subset for $(H,\Alt_f(\Omega))$, and $r$ the number of elements of $P$ and $r_1$ the number of elements of $P$ with finite support.  Then $\Omega_{H,f}$ is finite, and there exists an $H$-invariant partition $\Omega_{H,\infty} = \Omega_1 \cup \ldots \cup  \Omega_k$ of $\Omega_{H,\infty} $ into infinite subsets $\Omega_1, \ldots, \Omega_k$ with $k \leq \max (1,r_1)$, such that $H$ contains $\mathrm{Alt}_f(\Omega_1) \times \ldots \times \mathrm{Alt}_f(\Omega_k)$. If in addition $r_1=0$ then  $k=1$ and $| \Omega_{H,f}|  =  | \Omega \setminus \Omega_{1}| \leq r-1$. 
\end{prop}

\begin{proof}
	According to Lemma \ref{lem-conf-ht}, the subgroup $H$ has finitely many orbits in $\Omega$. Let $O_1,\ldots,O_\ell$ be the infinite orbits of $H$. By Lemma \ref{lem-conf-ht} again, every $H$-invariant partition of $\Omega_{H,\infty}$ that does not restrict to the trivial partition into singletons to any $O_i$, has cardinality at most $\max (1,r_1)$. Let $\Omega_{H,\infty} = \cup_{i,j} O_{i,j}$ be such a partition with maximal cardinality, where $O_i = \cup_j O_{i,j}$ for all $i$. Since the partition $O_i = \cup_j O_{i,j}$ cannot be further refined, the stabilizer of $O_{i,j}$ in $H$ acts primitively on $O_{i,j}$. Note that this implies in particular that $H$ cannot be abelian (since  a maximal subgroup in an abelian group must have finite prime index). Note also that Lemma \ref{lem-conf-ht} also ensures $| \Omega \setminus \Omega_{H,\infty}| \leq r-1$ when $r_1 = 0$.

	The core of the proof consists in showing that for each block $O_{i, j}$ the subgroup $H$ contains a non-trivial element whose support is finite and contained in $O_{i, j}$. We proceed by contradiction, and assume that $O_{i, j} = O$ is a block that does not satisfy the desired property. According to Lemma \ref{l-abelian-confined} and the above remark that no subgroup that is confined by $\Alt_f(\Omega)$ can be abelian, we may find a subset $Q$ that is confining for $(H,\Alt_f(\Omega))$ such that $Q$ has no element of order two. We write $Q = Q_0 \sqcup Q_1$, where $Q_0$ is the set of elements of $Q$ such that $\sigma^2$ has finite support, and write $Q_1 = \left\{\sigma_1,\ldots,\sigma_q\right\}$. 
	
	According to Lemma \ref{lem-displconf-perm}, there exist infinite subsets $\Sigma_1,\ldots,\Sigma_q \subset \Omega$ that form a displacement configuration for $\left\{\sigma_1,\ldots,\sigma_q\right\}$. For each $i$ we choose a finite subset $F_i \subset \Sigma_i$ of cardinality $q+5$. Upon replacing $Q$ by a conjugate by an element of $\Alt_f(\Omega)$ (which is also a confining subset for $(H,\Alt_f(\Omega))$), we may assume that $\mathrm{supp}(\sigma^2) \subset O$ for all $\sigma \in Q_0$ and also that $F_i \subset O$ for all $i$. Since $Q$ is confining for $(H,\Alt_f(\Omega))$, clearly $Q$ is confining for $(H,\Alt_f(O))$. We claim that actually $Q_1$ is confining for $(H,\Alt_f(O))$. Indeed $gQ_0g^{-1} \cap H$ is empty for all $g \in \Alt_f(O)$, because otherwise if $g \sigma g^{-1}$ belongs to $H$ with $ \sigma \in Q_0$, then $g \sigma^2g^{-1}$ would be a non-trivial element of $H$ whose support is finite and contained in $O$, which contradicts our assumption on the block $O$. Hence $Q_1$ is confining for $(H,\Alt_f(O))$.

	Now by construction $F_1,\ldots,F_q$ form a displacement configuration for $Q_1$. Since $|F_i| = q+5 \geq 5$, every non-trivial element of $\Alt_f(F_i)$ has a conjugacy class of size $>q$. Hence it follows that all the assumptions of Theorem \ref{thm-conf-abstract} are satisfied. By applying the theorem (with $r=q$), we deduce that there is $i$ such that $H$ contains a non-trivial subgroup of $\Sym(F_i)$. In particular $H$ contains non-trivial elements with finite support in $O$ since $F_i \subset O$. We have therefore reached a contradiction with the definition of $O$. 
	
	It now follows easily that $H$ contains $\mathrm{Alt}_f(O_{i,j})$ for all $i,j$. Indeed, fix a non-trivial element $h_{i,j} \in H$ whose support is finite  and contained in $O_{i, j}$. Then the  stabilizer of $O_{i, j}$ in $H$ acts primitively on $O_{i,j}$, and its image in $\Sym(O_{i, j})$ contains the element of finite support $h_{i,j}$, so it follows from Jordan-Wielandt theorem  that this image actually contains $\mathrm{Alt}_f(O_{i,j})$ (Theorem  \ref{thm-Jor-Wie}). Now the rigid stabilizer $\rist_H(O_{i,j})$ is non-trivial  (it contains $h_{i,j}$) and is normalized by $\mathrm{Alt}_f(O_{i,j})$, so it follows that $H$ contains $\mathrm{Alt}_f(O_{i,j})$, as desired.
\end{proof}

\begin{proof}[Proof of Theorem \ref{thm-confi-SYM}]
$\ref{item-confbysym} \implies \ref{item-confbyalt}$ is trivial. $\ref{item-confbyalt} \implies \ref{item-nicesbgp}$ follows from Proposition \ref{prop-confsym-precise}. For $\ref{item-nicesbgp}  \implies   \ref{item-confbysym}$ assume that $H$ satisfies the condition in \ref{item-nicesbgp}. Consider a finite subset $F$ of $\Omega$ of cardinality $2k+1+|\Omega_{H, f}|$, and let $P$ be the set of $3$-cycles with support in $F$. We claim that $P$ is a confining subset for $H$. Indeed for every $g\in \Sym(\Omega)$, the intersection of $g(F)$ with $\Omega_{H, \infty}$ must have cardinality at least $2k+1$, and hence there exists $i$ such that $|g(F)\cap \Omega_i|\ge 3$. It follows that $gPg^{-1} \cap \Alt_f(\Omega_i)$ is not empty, and hence $gPg^{-1}\cap H$ is not empty either. 
\end{proof}

\begin{remark}
Theorem \ref{thm-confi-SYM} provides a description of the confined subgroups for an arbitrary partially finitary group $G$: a subgroup $H\le G$ is confined if and only if $H$ satisfies condition \ref{item-nicesbgp} of the theorem. A well-studied family of partially finitary groups are Houghton's groups $H_n$. Recall that if $\Gamma_n$ is the graph obtained by gluing at one point $n$ infinite one-sided rays $\Delta_1,\cdots, \Delta_n$, then $H_n$ is the group of all permutations of the vertex set of $\Gamma_n$ which, outside of a finite set, coincide with a translation on each ray $\Delta_i$. By the above result we have a description of all confined subgroups of the groups $H_n$.
\end{remark}

\subsection{Application to high transitivity} \label{s-sym-ht}

The following lemma shows that confined subgroups behave well with respect to a dense embedding of a group $G$ into a topological group $S$.

\begin{lem} \label{lem-conf-dense-embed}
Let $\varphi: G \to S$ be an injective group homomorphism from a group $G$ into a topological group $S$ such that the image of $G$ in $S$ is dense. If $H$ is a confined subgroup of $G$, then $K = \overline{\varphi(H)}$ is a confined subgroup of $S$, and if $P$ is a confining subset for $(H,G)$ then so is $\varphi(P)$ for $(K,S)$. 
\end{lem}

\begin{proof}
	We have a $G$-equivariant map $\psi: G/H \to S/K$ with dense image. That $P$ is confining for $(H,G)$ means that the union $\cup_{\sigma \in P} \mathrm{Fix}_{G/H}(\sigma)$ is equal to the entire $G/H$. Here $\mathrm{Fix}_{G/H}(\sigma)$ is the set of fixed points of $\sigma$ for the left action of $G$ on $G/H$.  The coset space $S/K$ being Hausdorff,  the subset $\cup_{\sigma \in P} \mathrm{Fix}_{S/K}(\psi(\sigma))$ is closed in $S/K$, and it contains $\cup_{\sigma \in P} \psi(\mathrm{Fix}_{G/H}(\sigma)) = \psi(G/H)$. Therefore it is also dense, and hence it is equal to $S/K$. So $\varphi(P)$ is confining for $(K,S)$.
\end{proof}

The following covers Theorem \ref{thm-intro-main} from the introduction, which corresponds to the case where $G$ is not partially finitary.

\begin{thm} \label{thm-conf-ht-precise}
	Suppose that a group $G$ admits a faithful and highly transitive action on a set $\Omega$. Suppose that $H$ is a confined subgroup of $G$, that $P$ is confining for $(H,G)$, and let $r$ be the number of elements of $P$ and $r_1$ the number of elements of $P$ with finite support in $\Omega$. Then the following hold:
	\begin{enumerate}[label=\roman*)]
		\item \label{item-finorb} $\Omega_{H,f}$ is finite. Moreover if $r_1=0$ then $\Omega_{H,f}$ has cardinality at most $r-1$.
		\item  \label{item-infinorb} There exist infinite subsets $\Omega_1, \ldots,  \Omega_k $ with with $k \leq \max (1,r_1)$ such that $\Omega_{H,\infty} = \Omega_1 \cup \ldots \cup  \Omega_k$, and if $H_i$ is the stabilizer of $\Omega_i$ in $H$, then the action of $H_i$ on $\Omega_i$ is highly transitive  for all $i$.
	\end{enumerate}
So if $G$ is not partially finitary, then $\Omega_{H,f}$ has cardinality at most $r-1$, and the action of $H$ on $\Omega \setminus \Omega_{H,f}$ is faithful and highly transitive
\end{thm}

\begin{proof}
We write $S = \Sym(\Omega)$, $\varphi: G \to S$ the homomorphism associated to the action of $G$ on $\Omega$, and $K = \overline{\varphi(H)}$. Note that $H$ and $K$ have the same orbits in $\Omega$. By Lemma \ref{lem-conf-dense-embed} the subset $\varphi(P)$ is confining for $(K,S)$. Hence we are in position to apply Proposition \ref{prop-confsym-precise}. It follows from the proposition that $\Omega_{K,f}$ is finite, and there exist infinite subsets $\Omega_1, \ldots,  \Omega_k $ with with $k \leq \max (1,r_1)$ such that $\Omega_{K,\infty} = \Omega_1 \cup \ldots \cup  \Omega_k$, and $K$ contains $\Sym(\Omega_i)$ for all $i$ (since $K$ is closed). Therefore $\Omega_{H,f} = \Omega_{K,f}$ is indeed finite, and the action of $H_i$ on $\Omega_i$ is  highly transitive. Finally when $r_1=0$ we indeed have $k=1$ and $| \Omega_{H,f}|  \leq r-1$ by the proposition. This shows \ref{item-finorb} and \ref{item-infinorb}. 

For the last statement, if $G$ is not partially finitary then we may very well remove the elements of $P$ with finite support, and obtain a subset that is still confining for $(H,G)$.  By the previous paragraph the action of $H$ on $\Omega_{H,\infty}$ is highly transitive, and it is also faithful since otherwise $G$ would contain non-trivial elements with finite support, which is not the case by assumption. 
\end{proof}

\begin{remark} \label{remark-fullgp}
The bound $| \Omega_{H,f}|  \leq r-1$ in Theorem  \ref{thm-conf-ht-precise} is optimal. In order to see this, consider a group $\Gamma$ of homeomorphisms of the Cantor space $Y$ acting minimally on $Y$, and let $G$ be the topological full group of $\Gamma$. We fix a point $y \in Y$, and denote by $\Omega$ the orbit of $y$. Then one easily checks that the action of $G$ on $\Omega$ is faithful and highly transitive. Moreover the stabilizer subgroup $H$ of $y$ is confined in $G$, with a confining subset of cardinality $2$ (indeed if $P$ consists of two non-trivial elements with disjoint supports, then for every $g$ at least one element of $gPg^{-1}$ must fix $y$, and hence $H\cap gP g^{-1}\neq \varnothing$). So in this example $H$ has exactly one fixed point  in $\Omega$. Similarly by considering the stabilizer $H$ of $r-1$ fixed points in $G$, which admits a confining subset of cardinality $r$, we see that the bound $| \Omega_{H,f}|  \leq r-1$ is optimal.
\end{remark}

\begin{remark}
In the case where $G$ is partially finitary, the action of $H$ on $\Omega_{H,\infty}$ need not be highly transitive, and it is indeed necessary to pass to a partition of $\Omega_{H,\infty}$. This is  illustrated by the confined subgroups of the group $\Alt_f(\Omega)$ described in Theorem \ref{thm-confi-SYM}. 
\end{remark}

\section{Application to groups of PL-homeomorphisms} \label{sec-appl}

In this section we give applications of Corollary \ref{cor-intro-criteria} to groups of piecewise linear homeomorphisms of suspension flows. First in \S \ref{subsec-PL} we prove an auxiliary result on groups of piecewise linear homeomorphisms of an interval. In \S \ref{s-PL-flows} we recall the setting of \cite{MB-T-flows} and prove Theorem \ref{t-intro-PL}. Finally in \S \ref{subsec-variant} we discuss a family of examples of groups covered by this theorem and prove that they do not satisfy the obstructions to high transitivity described in \S \ref{elementary-obs}. 

\subsection{Subquotients of $PL(I)$} \label{subsec-PL}

In the sequel we let $I=[0, 1]$ be the compact unit interval in the real line, and denote by $\PL(I)$ the group of piecewise linear orientation preserving homeomorphisms of $I$. These are the homeomorphisms $g$ of $I$ fixing the extremities of $I$ and such that there are $0= x_0 < \ldots < x_n = 1$ and real numbers $a_0,\ldots, a_{n-1}, b_0,\ldots, b_{n-1}$ such that for all $i$ and all $x \in [x_i,x_{i+1}]$ one has $g(x) = a_i x + b_i$. 

Recall that Brin and Squier have shown that the group $\PL(I)$  does not contain any non-abelian free subgroup \cite{Brin-Squier}. Following a similar strategy, we prove a slight generalisation of this result, which places constraints on the subquotients of $\PL(I)$. Recall also that two subgroups $M,N$ of a group $G$ are \textbf{commuting subgroups} if every element of $M$ commutes with every element of $N$. Given a group $Q$ we denote by $Q'$ its commutator subgroup. 


\begin{prop} \label{prop-metab-subq}
	Let $Q$ be a non-trivial finitely generated subquotient of $\PL(I)$. Suppose that $Q$ does not admit two non-trivial commuting normal subgroups. Then $Q'$ is not finitely generated.
\end{prop}

Note that groups containing non-abelian free subgroups admit all countable groups as subquotients, so this restriction on the subquotients of $\PL(I)$ is indeed formally a generalisation of the Brin--Squier result. Before going into the proof, we isolate the following consequence.

\begin{cor} \label{c-PL-subquotient}
	Every finite subquotient of $\PL(I)$  is solvable.
\end{cor}

\begin{proof}[Proof of Corollary \ref{c-PL-subquotient}]
	Assume by contradiction that 	$Q$ is a non-solvable finite subquotient of $\PL(I)$ whose cardinality $|Q|$ is minimal. Let $N$ be a maximal proper normal subgroup of $Q$ and $S=Q/N$ the associated simple quotient. Suppose $S$ is not abelian. Then one can apply Proposition \ref{prop-metab-subq} to $S$, and we deduce that $S'=S$ is not finitely generated. Since $S$ is finite by assumption, this is absurd. Hence $S$ must be abelian of prime order. Since we chose $Q$ with minimal cardinality, $N$ must be solvable, and thus $Q$ is solvable as well, which is a contradiction. \end{proof}

Let us now turn to the proof of Proposition \ref{prop-metab-subq}.
Recall from \S \ref{s-notations-stab} that we define the support of an element $g\in \PL(I)$ as the set $\supp(g)=\{x\in I \colon gx \neq x\}$. Note that $\supp(g)$ is an open set; and moreover the definition of $\PL(I)$ implies that $\supp(g)$ has only finitely many connected components. Given  $H\le \PL(I)$, we define its support as $\supp(H):=\bigcup_{h\in H} \supp(h)$.  When $H$ is finitely generated it is enough to take the union for $h$ in a finite generating set, so in this case $\supp(H)$ has only finitely many connected components.

\begin{lem} \label{lem-reduce-compo}
	Let $Q$ be a finitely generated subquotient of $\PL(I)$. Suppose that $Q$ does not admit two non-trivial commuting normal subgroups. Then there exists a subgroup $L$ of $\PL(I)$ that admits $Q$ as a quotient and such that the support of $L$ is an open interval.
\end{lem}

\begin{proof}
Let $H$ be a subgroup of $\PL(I)$ that admits $Q$ as a quotient and such that the number $n$ of connected components of $\supp(H)$ is minimal. Observe that $Q$ is the quotient of a finitely generated subgroup of $\PL(I)$, so $n<\infty$.  Suppose for a contradiction that the statement is not true, that is to say $n \geq 2$. Let $J$ be a connected component of $\supp(H)$, and denote by $H_1$ the subgroup of $\PL(I)$ induced by the action of $H$ on $J$, and by $N_1$ the subgroup of $H$ consisting of elements that are supported in the complement of $J$. We also denote by $H_2$ the subgroup of $\PL(I)$ induced by the action of $H$ on the complement of $J$, and by $N_2$ the subgroup of $H$ consisting of elements that are supported in  $J$. Note that $H/N_i$ is isomorphic to $H_i$ for $i =1,2$.
	
	Let $N$ be a normal subgroup of $H$ such that $H/N$ is isomorphic to $Q$. By the choice of $H$, neither $H_1$ nor $H_2$ can admit $Q$ as a quotient since the number connected components of their support is $<n$. That means that neither $N_1$ nor $N_2$ is contained in $N$. Since  $N_1$ and $N_2$ commute, it follows that their images in $Q$ are commuting non-trivial normal subgroups, which is impossible by our assumption. Hence we have reached a contradiction, and hence $n$ must equal to $1$.
\end{proof}

The following lemma is based on the same argument as in \cite{Brin-Squier}. 

\begin{lem} \label{lem-one-compo}
	Let $H$ be a subgroup of $\PL(I)$ such that the support of $H$ is an open interval. Suppose that $H$ admits the group $Q$ as a quotient. Then for every finitely generated subgroup $K \leq Q'$, there exists $g \in Q$ such that $gKg^{-1}$ and $K$ commute. In particular if $Q'$ is finitely generated, then the group $Q$ is metabelian. 
\end{lem}

\begin{proof}
	Assume that $\supp(H)=(c,d)$, with $c<d$. Let $\pi: H \to Q$ be a surjective group homomorphism. Suppose $q_1,\ldots,q_n$ are generators of $K$. Since $K$ lies in $Q'$, one can find $h_1,\ldots,h_n$ in $H'$ such that $\pi(h_i)=q_i$ for all $i$. Since $h_i$ is in $H'$, the derivative of $h_i$ at the extreme points $c$ and $d$ must be equal to 1. Thus the closure of $\supp(h_i)$ is contained in $\supp(H)$, so that we can find $c<x <y<d$  such that for all $i$ one has $h_i(t)=t$ for all $t \leq x$ and all $t \geq y$. Since $H$ acts without fixed points in $(c,d)$, one can find $g \in H$ such that $g(x) \geq y$, and it follows that $g(\supp(h_i))$ is disjoint from $\supp(h_j)$ for all $i,j$. In particular $gh_ig^{-1}$ commutes with $h_i$, and hence  $\pi(g) K \pi(g)^{-1}$ and $K$ commute. This shows the first statement, and the second statement follows by taking $K = Q'$.
\end{proof}

\begin{proof}[Proof of Proposition \ref{prop-metab-subq}]
	Since $Q$ is not trivial and $Q$ does not admit a non-trivial abelian normal subgroup, the group $Q$ is not metabelian, and the statement then follows from Lemmas \ref{lem-reduce-compo} and \ref{lem-one-compo}.
\end{proof}

\subsection{Groups of $PL$-homeomorphisms of suspension flows} \label{s-PL-flows}

We start by recalling the setting of \cite{MB-T-flows}. Throughout this section we let $X$ be a Stone space, i.e. a compact totally disconnected space, and $\varphi\colon X\to X$ be a  homeomorphism. The dynamical system $(X, \varphi)$ will be called a \textbf{Stone system}.  We denote $Y^\varphi$ the  suspension (or mapping torus) of  $(X, \varphi)$, which  is defined as the quotient 
\[Y^\varphi:=(X\times \R) /\Z\]
with respect to the action of $\Z$  on $X\times \R$ given by $n\cdot (x, t)=(\varphi^n(x), t-n)$. We denote by $\pi\colon X\times \R\to Y^\varphi$ the projection, and by $[x, t]:=\pi(x, t)$. The suspension is endowed with a flow
\[\Phi\colon \R\times  Y^\varphi\to Y^\varphi, \quad \Phi^t([x, s])=[x, s+t].\]

The group of \textbf{self-flow-equivalences} of $(X, \varphi)$ is the group $\Hsf(\varphi)$ of all homeomorphisms of $Y^\varphi$ which send every $\Phi$-orbit to a $\Phi$-orbit in an orientation preserving way. This is a well-studied object in symbolic dynamics, see \cite{BCE, BC, APP}. We let $\Hsf_0(\varphi)$ be its subgroup consisting of homeomorphisms isotopic to the identity.

\begin{remark} Note that elements of $\Hsf_0(\varphi)$ preserve all $\Phi$-orbits, since these are precisely the path-components of $Y^\varphi$. The converse is false, namely a homeomorphism of $Y^\varphi$ which preserves every $\Phi$-orbit is not necessarily in $\Hsf_0(\varphi)$ \cite{BCE}; however, this holds true in the important case where $(X, \varphi)$ is {minimal} (i.e. all its orbits are dense), see Aliste-Prieto and Petite \cite{APP}. We also mention the following equivalent characterisation of $\Hsf_0(\varphi)$, proven in \cite{BCE}: a homeomorphism $g$ of $Y^\varphi$ belongs to $\Hsf_0(\varphi)$ if and only if there exists a continuous function $\alpha\colon Y^\varphi\to \R$ such that $g(y)=\Phi^{\alpha(y)}(y)$ for all   $y \in Y^\varphi$.
\end{remark}

Given a clopen set $C\subset X$, its \textbf{smallest return time} is :
\[\tau_C:=\min\{n\ge 1\colon \varphi^n(C)\cap C\neq \varnothing\}\in \N \cup\{\infty\}.\]

Assume that $I\le \R$ is an interval with length $|I|<\tau_C$. In this case we will say that the pair $(C, I)$ is $\varphi$-admissible; note that this is automatically the case if $|I|<1$. If $(C, I)$ is $\varphi$-admissible, then the restriction $\pi_{C, I}:=\pi|_{C\times I}$ is injective, and is called a \textbf{chart}. We denote by $U_{C, I}\subset Y^\varphi$ its image, which by abuse of terminology will also be called a chart. 

A homeomorphism $f\colon I\to J$ between two bounded intervals $I, J\subset \R$ is said to be PL if it is piecewise linear, with finitely many discontinuity points for the derivative. Moreover, we say that $f$ is \textbf{PL-dyadic} if all the discontinuity points of its derivative are dyadic rational, and in restriction to each interval of continuity of the derivative it is of the form $x\mapsto ax+b$, with $a=2^n, n\in \Z$, and $b\in \Z[\frac{1}{2}]$. Recall that Thompson's group $F$ is the group of all PL-dyadic homeomorphisms of the interval $(0, 1)$, and Thompson's group $T$ is the group of all PL-dyadic homeomorphisms of the circle $\R/\Z$. 
\begin{defin}[\cite{MB-T-flows}] \label{d-PL-phi}
	We let $\PL(\varphi)$ be the subgroup of $\Hsf_0(\varphi)$ consisting of all elements $g$ with the following property: for every $y\in Y^\varphi$, there exist a clopen set $C\subset X$, two dyadic intervals $I, J$ such that $(C, I)$ and $(C, J)$ are $\varphi$-admissible, and a PL-homeomorphism $f\colon I\to J$ such that $y\in U_{C, I}, g(U_{C, I})=U_{C, J}$ and the restriction of $g$ to $U_{C, I}$ is given in coordinates by 
	\[\pi_{C, J}^{-1} \circ g \circ \pi_{C, I}= \operatorname{Id}\times f\colon C\times I\to C\times J.\] 
	We further let $\Tsf(\varphi)$ be the subgroup of $\PL(\varphi)$ defined in a similar way but with the additional requirement that $f$ is PL-dyadic. 
\end{defin}

A proof of the following result can be found in \cite{MB-T-flows}.
\begin{thm}
	If $\varphi$ is a minimal homeomorphism of $X$, then $\Tsf(\varphi)$ is a simple group. If $(X, \varphi)$ is conjugate to a subshift, then $\Tsf(\varphi)$ is finitely generated.
\end{thm}

For every $\varphi$-admissible pair $(C, I)$, we can define a subgroup $\PL_{C, I}\le \PL(\varphi)$ isomorphic to $\PL(I)$ and supported in the chart $U_{C, I}$. The group $\PL_{C, I}$ is defined as the group of all homeomorphisms that act as the identity on the complement of $U_{C, I}$ and that  are given in coordinates by $\operatorname{Id}\times f$ on $U_{C, I}$ for some $f\in \PL(I)$. We also let $F_{C, I}$ be the subgroup of $\PL_{C, I}$ defied by the additional requirement that $f$ is PL-dyadic.   Note that $F_{C, I}$ is always isomorphic to the group  $F_I$ of PL-dyadic homeomorphisms of $I$; in particular if  the endpoints of $I$ are dyadic rationals then $F_I$ is isomorphic to Thompson's group $F:=F_{(0, 1)}$. It is shown in \cite{MB-T-flows} that the groups $F_{C, I}$ generate $\Tsf(\varphi)$ when $(C, I)$ varies over all admissible pairs. 

The following lemma is essentially Lemma 7.2 in \cite{MB-T-flows}. It analyses the structure of the germ stabilizers $\PL(\varphi)_y^0$ for $y\in Y^\varphi$ when $\varphi$ is minimal. 

\begin{lem} \label{l-first-return} Assume that $\varphi$ is a minimal homeomorphism of $X$, and fix $y\in Y^\varphi$. Let $S\subset \PL(\varphi)_y^0$ be a finite subset. Then there exist disjoint charts $U_{C_1, I_1},\ldots, U_{C_n, I_n}$ whose closure does not contain $y$ such that $S$ is contained in the subgroup $\langle \PL_{C_1, I_1},\ldots, \PL_{C_n, I_n}\rangle\simeq \PL({I_1})\times \cdots \times \PL({I_n})$. \end{lem}

\begin{proof} The analogous statement for the group $\Tsf(\varphi)$ is \cite[Lemma 7.2]{MB-T-flows}, and the same proof works for the whole group $\PL(\varphi)$. For completeness we recall a sketch of the argument.   Since $S$ is a finite subset of $\PL(\varphi)_y^0$, its elements must act trivially on a chart $U_{C, J}\ni y$ for some interval $J=(a, b)$ with $|J|=b-a<1$.  Consider the first return time function $T_C\colon C\to \N_{>0}$, namely $T_C(x):=\min\{n\ge 1\colon \varphi^n(x)\in C\}$. By minimality, this function is bounded and locally constant on $C$, so we can find a clopen partition $C=C_1\sqcup\cdots \sqcup C_n$ such that $T_C\equiv t_i$ on $C_i$, with $t_1,\cdots, t_n\in \N_{>0}$. For each $i$ consider the interval $I_i=(b, t_i -a)$, and note that $|I_i|=t_i-|J|$. By construction we have $Y^\varphi\setminus \overline{U}_{C, J}=\bigsqcup_{i=1}^nU_{C_i, I_i}$, and the group generated by $S$ preserves each chart $U_{C_i, I_i}$. After taking a refinement of the partition $C=C_1\sqcup\cdots\sqcup C_n$ if necessary, we reach the conclusion of the lemma. \end{proof}

We are now in position to prove Theorem \ref{t-intro-PL} from the introduction, which is recalled here.

\begin{thm} \label{t-PL-flow-not-ht}
	Let $(X, \varphi)$ be a minimal Stone system, and let $G\le \PL(\varphi)$ be any subgroup whose action on $Y^\varphi$ is topologically nowhere free.   Then $G$ does not admit any faithful highly transitive action. In particular this holds true for the group $G=\Tsf(\varphi)$.  
\end{thm}

\begin{proof}
Note that the group $\PL(\varphi)$ is torsion-free, and hence does not contain any finitary alternating group. The germ stabilizers $G_y^0$ are confined in $G$ since the action is topologically nowhere free (Lemma \ref{lem-conf-stab}). Moreover every finitely generated subgroup of $G_y^0$ is isomorphic to a subgroup of $\PL(I)$ by Lemma \ref{l-first-return}, and hence has all its finite subquotients solvable by Corollary \ref{c-PL-subquotient}. Therefore $G_y^0$ cannot admit any faithful highly transitive action (see \S \ref{elementary-obs}). The conclusion then follows from Corollary \ref{cor-intro-criteria}.
\end{proof}

\subsection{A variant on a non-orientable quotient of $Y^\varphi$} \label{subsec-variant}

It is natural to compare Theorem \ref{t-PL-flow-not-ht} with the criteria to exclude high transitivity of a group $G$ described in \S \ref{elementary-obs}. For instance the groups $\Tsf(\varphi)$ always contain non-abelian free subgroups \cite[Corollary 8.2]{MB-T-flows}, and thus do not satisfy any restriction on their subquotients. Moreover their simplicity rules out the existence of a similar restriction in a non-trivial normal subgroup. It is not clear whether $\Tsf(\varphi)$ can satisfy a non-trivial mixed identity. In this subsection, we describe a related family of subgroups of $\PL(\varphi)$ covered by Theorem \ref{t-PL-flow-not-ht}, which retain the same main properties, and for which we can exclude this as well. This family is inspired by Hyde and Lodha's examples of finitely generated simple groups acting on the real line \cite{HL}, which can be embedded in $\PL(\varphi)$ for a suitable $(X, \varphi)$. In fact, relying on subsequent work of the same authors and  Navas and Rivas \cite{HLNR}, one can show that the Hyde--Lodha examples arise as special cases of the construction described here (see Remark \ref{r-HLNR} below). 

\medskip

The main idea is to consider a   2-to-1 quotient of the space  $Y^\varphi$ in such a way that the natural partition of $Y^\varphi$ into $\Phi$-orbits passes to the quotient, but no longer  admits an orientation coming from a globally defined flow.  To this end we let again $X$ be a Stone space, and consider an action $D_\infty\acts X$ of the infinite dihedral group $D_\infty:=\Z\rtimes \Z/2\Z$. Such an action is given by a pair of homeomorphisms $(\varphi, \sigma)$  of $X$  such that $\sigma^2=\operatorname{Id}$ and  $\sigma\varphi \sigma=\varphi^{-1}$. Explicitly an element $(n, j)\in D_\infty=\Z\rtimes \Z/2\Z$ acts on $x$ by $(n, j)\cdot x=\varphi^n\sigma^j(x)$. 
We call the triple $(X, \varphi, \sigma)$ a \textbf{Stone $D_\infty$-system}. 

\emph{Throughout this section, we make the standing assumption that the action of $D_\infty$ defined by $(\varphi, \sigma)$ is free.}

Consider the  diagonal action of  $D_\infty$  on $X\times \R$, where  the action on $\R$ is the natural action by isometries given by  by $(n, j)\cdot t:=(-1)^jt-n$ for $(n, j)\in D_\infty, t\in \R$.
We define a space $Y^{\varphi, \sigma}$ as the quotient
\[Y^{\varphi, \sigma}=(X\times \R)/D_\infty.\]
We let again $\pi\colon X\times \R \to Y^{\varphi, \sigma}$ be the quotient map, and set $[x, t]=\pi(x, t)$.

Note that given $(X, \varphi, \sigma)$  we  can also consider the Stone system $(X, \varphi)$ obtained by dropping $\sigma$.  Then  $\sigma$  naturally induces a homeomorphism $\hat{\sigma}$ of order 2 of $Y^\varphi$ given by: 
\begin{equation} \label{e-hatsigma}\hat{\sigma}\colon Y^\varphi\to Y^\varphi, \quad \hat{\sigma}([x, t])=[\sigma(x), -t],\end{equation}
and we have $Y^{\varphi, \sigma}=Y^\varphi/\langle\hat{\sigma}\rangle$.

  For $x\in X$ the restriction of $\pi$ to $\{x\}\times \R$ is injective, and we denote $\ell_x$ its image, and let $\gamma_x\colon \R\to \ell_x$ be the natural parametrisation
\begin{equation} \label{e-leaves} \gamma_x\colon \R \to \ell_x, \quad \gamma_x(t)=\pi(x, t).\end{equation}

 Subsets  $\ell\subset Y^{\varphi, \sigma}$ of the form $\ell=\ell_x$ for some $x$ will be called \textbf{leaves}. Note that $Y^{\varphi, \sigma}$ is a disjoint union of leaves, which coincide with the path-connected components of $Y^{\varphi, \sigma}$. However note that the parametrisation \eqref{e-leaves}  depends on the choice of a point $x\in X$ such that $\ell=\ell_x$; a different choice of $x$ amounts to precomposing $\gamma_x$   by an element of $D_\infty$. Since the action of $D_\infty$ on $\R$ contains reflections, leaves \emph{do not} have a well-defined preferred orientation.

Given a clopen subset $C\subset X$ and an interval $I$, we say that the pair $(C, I)$ is $(\varphi, \sigma)$-\textbf{admissible} if the restriction of $\pi$ to $C\times I$ is injective; equivalently if $C\times I$ is disjoint from its image under all non-trivial elements of $D_\infty$. 
If $(C, I)$ is an admissible pair, we denote by $\pi_{C, I}$ the restriction of the quotient map $\pi$ to $C\times I$,  and by $U_{C, I}\subset Y^{\varphi, \sigma}$ its image. In this context, the map $\pi_{C, I}$, or simply its image $U_{C, I}$, will be called a \textbf{chart}.

The following lemma guarantees that charts form indeed a basis for the topology of $Y^{\varphi, \sigma}$, and moreover a bounded segment  on a leaf can always be ``thickened'' to a chart. 

\begin{lem} \label{l-charts-basis} For every  $(x, t)\in X\times \R$, and for every bounded  interval $I\subset \R$ containing $t$, there exists a clopen neighbourhood $C\subset X$  of $x$ such that $(C, I)$ is $(\varphi, \sigma)$-admissible.
\end{lem}

\begin{proof} Since the isometric action of $D_\infty$ on $\R$ is proper, the set $F:=\{g\in D_\infty\setminus\{1\}\colon g\cdot I\cap I\neq \varnothing\}$ is finite.  Since moreover the action of $D_\infty$ on $x$ is free, we have $g(x)\neq x$ for every $g\in F$. Thus there exists a clopen neighbourhood $C$ of $x$ such that $g(C)\cap C=\varnothing$ for every $g\in F$. Now for $g\in D_\infty$ the set $g\cdot (C\times I)=(g\cdot C) \times (g\cdot I)$,  is disjoint from $C\times I$: if $g\notin F$ this follows looking at the second factor, and for $g\in F$ it follows by looking at the first factor. Thus the pair $(C, I)$ is $(\varphi, \sigma)$-admissible. \qedhere
\end{proof}

The following definition is analogous to Definition \ref{d-PL-phi}.

\begin{defin}
	We let  $\PL(\varphi, \sigma)$ be the subgroup of $\Hsf_0(\varphi, \sigma)$ consisting of elements $g$ such that for every $y\in Y^{\varphi, \sigma}$, there exist $(\varphi, \sigma)$-admissible pairs $(C, I)$ and $(C, J)$ such that  $y\in U_{C, I}$, $g(U_{C, I})=U_{C, J}$ and there exists a PL homeomorphism $f\colon I\to J$ such that the restriction of $g$ to $U_{C, I}$ is given in coordinates by
	\[\pi_{C, J}\circ g \circ \pi_{C, I}^{-1}=\operatorname{Id}\times f \colon C\times I\to C\times J.\]
	We let $\Tsf(\varphi, \sigma)$ be the subgroup of $\PL(\varphi, \sigma)$ defined by the additional requirement that $f$ is PL-dyadic.
\end{defin}

For every $(\varphi, \sigma)$-admissible pair, we have subgroups $\PL_{C, I}\le \PL(\varphi, \sigma)$ and $F_{C, I}\le \Tsf(\varphi, \sigma)$ defined analogously as in the previous subsection.

By viewing $Y^{\varphi, \sigma}=Y^{\varphi}/\langle \hat{\sigma}\rangle$ as in \eqref{e-hatsigma},  the group $\PL(\varphi, \sigma)$ can alternatively  be seen as a subgroup of  $\PL(\varphi)$. 
\begin{prop} \label{p-lift} The group $\PL(\varphi, \sigma)$ lifts to a group of homeomorphisms of $Y^\varphi$, namely the subgroup of $\PL(\varphi)$ consisting of elements that commute with $\hat{\sigma}$.
\end{prop}
\begin{proof} It is clear that every element of $\PL(\varphi)$ which commutes with $\hat{\sigma}$ descends to a homeomorphism of $Y^{\varphi, \sigma}$ that belongs to $\PL(\varphi, \sigma)$. Conversely, given $g\in \PL(\varphi, \sigma)$ we check that $g$ lifts to an element  $\tilde{g} \in \PL(\varphi)$ which commutes with $\hat{\sigma}$. By definition we can find two covers $Y^{\varphi, \sigma}=\bigcup_{i=1}^n U_{C_i, I_i}=\bigcup_{i=1}^n U_{C_i, J_i}$ by charts, and PL-homeomorphisms $f_i\colon I_i\to J_i, i=1,\ldots, n$, such that $g(U_{C_i, I_i})=U_{C_i, J_i}$ as in the definition of elements of $\PL(\varphi, \sigma)$. Each pair $(C_i, I_i)$ defines also a chart $U^\varphi_{C_i, I_i}\subset Y^\varphi$ (here we use the notation $U^\varphi_{C, I}$ for charts in $Y^\varphi$ to distinguish them from charts in $Y^{\varphi, \sigma}$). Now the fact that $(C_i, I_i)$ is a $(\varphi, \sigma)$-admissible pair for every $i$ is equivalent to the fact that $\hat{\sigma}(U^\varphi_{C_i, I_i})\cap U^\varphi_{C_i, I_i}=\varnothing$. Noting that $\hat{\sigma}(U_{C_i, I_i})=U_{\sigma(C_i), -I_i}$, and using that the charts $U_{C_i,I_i}$ and $U_{C_i, J_i}$ form covers of $Y^{\varphi, \sigma}$, we have two covers $Y^\varphi=\bigcup_{i=1}^n \left(U^\varphi_{C_i, I_i} \cup U^\varphi_{\hat{\sigma}(C_i), -I_i}\right)=\bigcup_{i=1}^n \left(U^\varphi_{C_i, J_i} \cup U^\varphi_{\hat{\sigma}(C_i), -J_i}\right)$. Then we can consider the element $\tilde{g}\in \PL(\varphi)$ which in restriction to each chart $U^\varphi_{C_i, I_i}$ is given in coordinates by the homeomorphism $f_i\colon I_i\to J_i$, and in restriction to each chart $U^\varphi_{\sigma(C_i), -I_i}$ is given by $\bar{f_i}\colon -I_i\to J_i$, where $\bar{f}(t):=-f(-t)$. It is routine to check that $\tilde{g}$ is a well-defined lift of $g$, and this gives the desired conclusion. \qedhere
	
\end{proof}

The goal of the remaining part of this section is to prove the following. 

\begin{thm}\label{t-phi-sigma}
	Let $(X, \varphi, \sigma)$ be a free Stone $D_\infty$-system such that $(X, \varphi)$ is minimal. Then the group $\Tsf(\varphi, \sigma)$ satisfies the following:
	\begin{enumerate}[label=(\roman*)]
	\item \label{i-phi-sigma-ht} it does not admit any faithful highly transitive action;
	\item \label{i-phi-sigma-mixid} it does not satisfy any non-trivial mixed identity;
	\item \label{i-phi-sigma-free} it contains non-abelian free subgroups. 
	\end{enumerate}
\end{thm}

We will prove each property separately. The first part is a direct application of Theorem \ref{t-PL-flow-not-ht}, thanks to Proposition \ref{p-lift}. 
\begin{proof}[Proof of Theorem \ref{t-phi-sigma} \ref{i-phi-sigma-ht}]
	Choose $\tilde{y}\in Y^\varphi$ which projects to $y$. By Proposition \ref{p-lift}, the group $G=\Tsf(\varphi, \sigma)$ lifts to an isomorphic subgroup $\tilde{G}\le \PL(\varphi)$. Moreover for every $y\in Y^{\varphi, \sigma}$, and every $\tilde{y}\in Y^\varphi$ which projects to $y$ we have $G_y^0= \tilde{G}^0_{\tilde{y}}$. Hence $\tilde{G}$ has non-trivial germ-stabilizers. The conclusion then follows from Theorem \ref{t-PL-flow-not-ht}.  \qedhere
\end{proof}

We now turn to the absence of non-trivial mixed identity. To this end, we first rule out the existence of a specific type of non-trivial mixed identities within the group $\homeo_0(\R)$ of orientation-preserving homeomorphisms of $\R$. For $f\in \homeo_0(\R)$, we say that $f$ is \textbf{unboundedly positive} if the set $\{t\in \R \colon f(t)>t\}$ accumulates to both $+\infty$ and $-\infty$. We say that $f$ is \textbf{unboundedly negative} if $f^{-1}$ is unboundedly positive.

Assume that $G$ is a group and $w$ is an element of the free product $\langle z \rangle *G$, where $z$ is a generator of an infinite cyclic group. For $h\in G$ we denote $w(h)\in G$ the image of $w$ under the homomorphism from $\langle z \rangle *G$ to $G$ which is the identity on $G$ and  maps $z$ to $h$. By definition the group $G$ satisfies the mixed identity $w=1$ if $w(h)=1$ for every $h\in G$.

 \begin{lem} \label{l-mix-id-homeo}
	Let $w\in \langle z \rangle * \homeo_0(\R)$ be a non-trivial element given by the reduced word $w=z^{n_k}g_k\cdots z^{n_1}g_1$, where $n_1,\cdots, n_k\neq 0$ and $g_1, \cdots, g_k\in \homeo_0(\R)$ are all unboundedly positive.  Then there exists $h\in \homeo_0(\R)$ such that $w(h)\neq 1$.
\end{lem} 
Note that the lemma is also true if we assume instead that  $g_1, \cdots, g_k$ are unboundedly negative. 
\begin{proof}
If $k=1$ the conclusion is obvious, so we can assume that $k\ge 2$.	Using the assumption, we can find points $t_1,\cdots, t_k\in \R$ such that 
	\[t_1<g_1(t_1)<t_2<g_2(t_2)<\cdots<t_k<g_k(t_k).\] Choose points $p_i\in (t_i, g_i(t_i))$ for $i=1,\cdots, k$. Choose an element $h\in \homeo_0(\R)$ which fixes $p_i$ for $i=1,\cdots, k$, acts trivially on the interval $(-\infty, p_1)$ and on $(p_k, +\infty)$, and such that for $i=1,\cdots, k$ it satisfies $h(g_i(t_i))>t_{i+1}$ if $n_i>0$, and $h(t_{i+1})<g_i(t_i)$ if $n_i<0$. Note that in both cases, we have $h^{n_i}(g_i(t_i))>t_{i+1}$ for $i=1,\cdots, k-1$. Then one readily checks that $w(h)(t_1)\geq t_k>t_1$. In particular $w(h)\neq 1$ as desired. \qedhere
\end{proof}

\begin{remark}
	If we drop the assumption on the parameters $g_1,\cdots, g_k$, the conclusion is false. Indeed the group $\homeo_0(\R)$ does satisfy non-trivial mixed identities \cite{id-F}.
\end{remark}

For every $x\in X$ the group $\Tsf(\varphi, \sigma)$ preserves the leaf $\ell_x\subset Y^{\varphi, \sigma}$. Its action on $\ell_x$ provides a representation $\rho_x\colon \Tsf(\varphi, \sigma)\to \homeo_0(\R)$ given by 
\[\rho_x(g)=\gamma_x^{-1}\circ g|_{\ell_x} \circ\gamma_x,\]
 where $\gamma_x$ is the parametrisation \eqref{e-leaves}. The representations $\rho_x$ have the following properties.
\begin{lem} \label{l-unboundedly-positive}
	Let $(X, \varphi, \sigma)$ be a free Stone $D_\infty$-system.  Fix $x\in X$, and consider the representation $\rho_x\colon \Tsf(\varphi, \sigma)\to \homeo_0(\R)$. Then:
	\begin{enumerate}[label=(\roman*)]
		\item \label{i-dense-image} The image of $\rho_x$ is dense in $\homeo_0(\R)$ endowed with the compact-open topology. 
		\item \label{i-unboundedly-positive} If $(X, \varphi)$ is minimal, then $\rho_x$ is faithful and the image of every non-trivial element is both unboundedly positive and unboundedly negative. 
	\end{enumerate}
\end{lem}

\begin{proof}
	\ref{i-dense-image} Let $h\in \homeo_0(\R)$. By definition of the compact-open topology, we need to show that for every compact interval $K\subset \R$ and every $\varepsilon>0$ there exists $g\in \Tsf(\varphi, \sigma)$ such that $|\rho_x(g)(t)-h(t)|\le \varepsilon$ for every $t\in K$.  Let $I\subset \R$ be a large enough bounded open interval  such that $0\in I$ and $K\cup h(K) \subset I$.   By Lemma \ref{l-charts-basis} we can find a clopen neighbourhood $C$ of $x$ such that the pair $(C, I)$ is $(\varphi, \sigma)$-admissible. Then the image under $\rho_x$ of the group $F_{C, I}\le \Tsf(\varphi, \sigma)$  preserves $I$ and its action induces the standard action of $F_I$. Then, using that $F_I$ is dense in $\homeo_0(I)$, we can find $g\in F_{C, I}$ such that $|\rho_x(g)(t)-h(t)|\le \varepsilon$ for every $t\in K$, as desired. 
	
	\ref{i-unboundedly-positive} If $(X, \varphi)$ is minimal, then every leaf $\ell_x$ is dense in $Y^{\varphi, \sigma}$, and thus the action of $\Tsf(\varphi, \sigma)$ on it is faithful. Choose $g\in \Tsf(\varphi, \sigma)$ non-trivial and $y\in \ell_x$ such that $g(y)\neq y$. Write $y=\gamma_x(t_0)$ for some $t_0\in  \R$.  Choose also a chart $U_{C, I}$ containing $y$, with $(x, t_0)\in C\times I$, such $g$ is given in coordinates by $\operatorname{Id}\times f$ for some PL-dyadic homeomorphism $f\colon I\to J$. Note that $g(y)\neq y$ implies that $f(t_0)\neq t_0$, so that  upon restricting $I$ we can assume without loss of generality that $f$ has no fixed points in $I$.  Moreover, after replacing $g$ by its inverse if necessary, we can assume that $f(t)>t$ for every $t\in I$. By construction, the restriction of $\rho_x(g)$ to $I$ coincides with $f$. 
	
	Consider now the following two sets of integers:
	\[E_+=\{n\in \Z \colon \varphi^n(x)\in C\}, \quad \quad E_-=\{n\in \Z \colon \varphi^n\sigma(x)\in C\}.\]
	By minimality of $(X, \varphi)$ both $E_+$ and $E_-$ contain sequences accumulating to both $+\infty$ and $-\infty$. For $n\in E_+$ let $I_n$ be the image of $I$ under the translation $t\mapsto t+n$ (that is which corresponds to the element $(n, 0)^{-1}\in D_\infty$). For $n\in E_-$ let ${I_n}$ be the image of $I$ under the map $t\mapsto -t-n$ (that is the element $(n, 1)^{-1}=(n, 1)\in D_\infty$). Note that the set $\bigcup_{n\in E_+\cup E_-} I_n $ is precisely the sets of times $t\in \R$ such that $\gamma_x(t)\in U_{C\times I}$. More precisely for $t\in I_n$ with $n\in E_+$ we have $\gamma_x(t)=\pi(\varphi^n(x), t-n)$, so that $\gamma_x(t)$ runs through the arc $\pi(\{\varphi^n(x)\} \times I)$ according to the orientation given by the orientation of $I$. Instead if $t\in I_n$ with $n\in E_-$ we have $\gamma_x(t)=\pi(\varphi^n\sigma(x), -t-n)$, so that $\gamma_x(t)$ runs through the arc $\pi(\{\varphi^n\sigma(x)\}\times I)$ in the opposite orientation. Now note that for $n\in E_+$ the restriction of $\rho_x(g)$ to the interval $I_n$ coincides with the conjugate of $f$ by the translation $t\mapsto t+n$, and thus satisfies $\rho_x(g)(t)>t$ for every $t\in I_n$. In contrast, when $n\in E_-$ the restriction of $\rho_x(g)$ to ${I}_n$ coincides with the conjugate of $f$ under the orientation reversing map $t\mapsto -t-n$, and therefore satisfies $\rho_x(g)(t)<t$ for $t\in I_n$. Since both sets $E_+$ and $E_-$ are unbounded from above and from below, we deduce that $\rho_x(g)$ is unboundedly positive and unboundedly negative. \qedhere
\end{proof}

\begin{proof}[Proof of Theorem \ref{t-phi-sigma} \ref{i-phi-sigma-mixid}]
	We write $G = \Tsf(\varphi, \sigma)$. According to \cite[Remark 5.1]{Hull-Osin}, it is enough to prove that there does not exist a non-trivial element $w \in  \mathbb{Z} \ast G$ such that $G$ satisfies the mixed-identity $w=1$.	 Assume that $w\in \langle z \rangle * G$ is a non-trivial reduced word of the form $w(z)=z^{n_k}g_k\cdots z^{n_1}g_1$, with $g_1, \cdots, g_k\in G$ such that $w(h)=1$ for every $h\in G$. 
	By Lemma \ref{l-unboundedly-positive} we can embed $G$ as a dense subgroup of $\homeo_0(\R)$ such that all non-trivial elements are unboundedly positive (and unboundedly negative). By density this implies that $w(h)=1$ for every $h\in \homeo_0(\R)$, which is in contradiction with Lemma \ref{l-mix-id-homeo}. \qedhere
	\end{proof}

We now wish to prove that the group $\Tsf(\varphi, \sigma)$ contains free subgroups. To this end, recall that an action of a group $G$ on a compact space $Y$ is said to be \textbf{extremely proximal} if for every proper closed subset $Z\subsetneq Y$, there exists a point $y\in Y$ such that for every neighbourhood $V$ of $y$ there exists $g\in G$ such that $g(Z)\subset V$. If an action of a group $G$ on a compact space is minimal and extremely proximal, a well-known ping-pong argument implies that $G$ admits non-abelian free subgroups \cite[Theorem 3.4]{Glas-EP}. Thus, it is enough to show the following.

\begin{prop}
Let $(X, \varphi, \sigma)$ be a minimal Stone $D_\infty$-system. Then the action of $\Tsf(\varphi, \sigma)$ on $Y^{\varphi, \sigma}$ is minimal and extremely proximal. 

\end{prop}
\begin{proof}
The minimality of $(X, \varphi, \sigma)$ implies that all leaves of $Y^{\varphi, \sigma}$ are dense. Since it is not difficult to see (for instance from part \ref{i-dense-image} in Lemma \ref{l-unboundedly-positive}) that the $\Tsf(\varphi, \sigma)$-orbit of every $y\in Y^{\varphi, \sigma}$ is dense in its leaf, it follows that the action of $\Tsf(\varphi, \sigma)$ on $Y^{\varphi, \sigma}$ is minimal.  To show that it is extremely proximal, we will use an  argument similar to the proof of Lemma \ref{l-first-return}. Let $Z\subsetneq Y^{\varphi, \sigma}$ be a closed subset, and fix $y\in Y^{\varphi, \sigma}\setminus Z$. Let $V$ be a neighbourhood of $y$ such that $V\cap Z=\varnothing$. Let $U_{C, I}$ be an open chart containing $y$ and such that $\overline{U_{C, I}}\subset V$. Then, using that every leaf that exits $U_{C, I}$ in both directions must return to it,  and reasoning in a similar way as in the proof of Lemma \ref{l-first-return}, we can find a decomposition of the complement  $Y^{\varphi, \sigma}\setminus \overline{U_{C, I}}$ as a disjoint union of open charts
\[Y^{\varphi, \sigma}\setminus \overline{U_{C, I}}=\bigsqcup_{i=1}^n U_{C_i, I_i}, \] 
such that the boundary of $\partial U_{C_i, I_i}$ is contained in the boundary of $\partial U_{C, I}$ (the only difference here from Lemma \ref{l-first-return}  is that it may happen that the two components of the boundary of $U_{C_i, I_i}$ corresponding to the two extreme points of $I_i$ are both contained in a same component of $\partial U_{C, I}$ corresponding to one of the endpoints of $I$). Set $Z_i=Z\cap U_{C_i, I_i}$, and $V_i=V\cap U_{C_i, I_i}$. Note that since $\partial U_{C_i, I_i}\subset V$, if we set $I_i=(a_i, b_i)$, then there exists $\varepsilon >0$ such that the set $V_i$ contains $U_{C_i, (a_i, a_i+\varepsilon)}$ and $U_{C_i, (b_i, b_i-\varepsilon)}$ for some $\varepsilon >0$, while the set $Z_i$ is contained in $U_{C_i, (a_i+\varepsilon, b_i-\varepsilon)}$. Using that the group $F_{I_i}$ acts without fixed points on $I_i$, we can find for every $i=1,\ldots, n$ an element $g_i\in F_{C_i, I_I}$ such that $g_i(Z_i)\subset V_i$.   Then the element $g=g_1\cdots g_n\in \Tsf(\varphi, \sigma)$ satisfies $g(Z)\subset Z$. Since $V$ was an arbitrary neighbourhood of $y$, the conclusion follows. \qedhere
\end{proof}
This  concludes the proof of Theorem \ref{t-phi-sigma}.

\begin{remark}
	Most other results on the groups $\Tsf(\varphi)$ from \cite{MB-T-flows} can be adapted to $\Tsf(\varphi, \sigma)$ without much effort. In particular if the Stone $D_\infty$-system $(X, \varphi, \sigma)$ is minimal then the group $\Tsf(\varphi, \sigma)$ is simple, and if the underlying Stone system $(X, \varphi)$ is conjugate to a subshift then the group $\Tsf(\varphi, \sigma)$ is finitely generated.  We do not explain the details of the proofs of these results, since they can be proven by a routine adaptation of the arguments in \cite{MB-T-flows}, by replacing charts in $Y^\varphi$ and the subgroups $F_{C, I}$ by their analogues for $\Tsf(\varphi, \sigma)$, modulo some mild modifications.\end{remark}

\begin{remark} \label{r-HLNR}
Here we elaborate on the connection between the family of groups $\Tsf(\varphi, \sigma)$ and the groups defined by Hyde and Lodha in \cite{HL}. The starting point of their constructions is a bi-infinite sequence $\rho=(u_n)_{n\in \Z}$, where each element $u_n$ belongs to the finite alphabet  $A=\{a, a^{-1}, b, b^{-1}\}$. The sequence $\rho$ is required to satisfy suitable conditions (see \cite{HL}), among which is the requirement that for every finite word $w=w_0\cdots w_n$ of consecutive elements of $\rho$, the formal inverse $w^{-1}=w_n^{-1}\cdots w_1^{-1}$ should also appear in $\rho$. To every such  sequence $\rho$, they associate a group $G_\rho\le \homeo_0(\R)$ generated by a finite set of explicit piecewise linear homeomorphisms of $\R$ with an infinite discrete set of discontinuity points for the derivative. Roughly speaking, each letter $u_n$ of $\rho$ prescribes how each generator acts in restriction to the interval $[n, n+2]$.  A more intrinsic description of which homeomorphisms of $\R$ belong to the group $G_\rho$ was later obtained by Hyde-Lodha-Navas-Rivas \cite[Theorem 0.8]{HLNR}. 
The condition that $\rho$ is closed under inverses ensures that elements of $G_\rho$ are both unboundedly positive and negative, which plays a  role in the combinatorial analysis in \cite{HL, HLNR}.

Now, given a sequence $\rho$  as above, let  $X\subset A^\Z$ be the subshift given by the orbit-closure of $\rho$ under the shift map, and let $\varphi$ be the restriction of the shift to $X$. The assumption that $\rho$ is closed under formal inverses implies that we can define an involution 
\[\sigma\colon X\to X, \quad \sigma ((v_n)_n)=(v_{-n-1}^{-1})_n.\]
One readily checks that the triple $(X, \varphi, \sigma)$ is a Stone $D_\infty$-system. The  
 description of the group $G_\rho$ provided by \cite[Theorem 0.8]{HLNR} essentially implies that $G_\rho$ coincides with the group $\Tsf(\varphi, \sigma)$ associated to this system. 
\end{remark}

\section{A classification result of highly transitive actions} \label{sec-classification}

In this section we explain how Theorem  \ref{thm-intro-main} can be used to classify all highly transitive actions of certain groups. Below for simplicity we focus of the case of the Higman--Thompson groups \cite{Hig-fp}, but the argument generalizes to other classes of groups, see Remark \ref{rmq-general-Vd}.

 \subsection{Preliminaries} \label{subsec-prelim-V}
 Throughout this section we let $X_d=\{0,\cdots,d-1\}^\N$ be the Cantor set of one sided sequences over a $d$-letter alphabet. Given a finite word $w\in \{0,\cdots, d-1\}^*$ we let $C_w\subset X_d$ be the cylinder subsets of sequences that admit $w$ as a prefix. Recall that the Higman--Thompson group $V_d$ is the group of all homeomorphisms of $X_d$ such that there exist two collections of finite words $w_1,\cdots, w_k, u_1,\cdots, u_k\in\{0,\ldots, d-1\}^*$ with $X_d=C_{w_1}\sqcup \cdots \sqcup C_{w_k}=C_{u_1}\sqcup\cdots \sqcup C_{u_k}$ and such that $g(w_i\xi)=u_i\xi$ for all $i=1\cdots k$ and all $\xi\in X_d$. 
 It follows easily from the definition that the $V_d$-orbits in $X_d$ coincide with the cofinality classes of sequences, meaning that two sequences $\xi=x_0x_1\cdots, \eta=y_0y_1\cdots$ are in the same $V_d$-orbit if and only if there exists $m,n \geq 0$ such that $x_{m+i}=y_{n+i}$ for all $i\ge 0$. Moreover the action of $V_d$ on each orbit is highly transitive. Note also that the action of $V_d$ on $X_d$ is topologically nowhere free, and hence the stabilizer in $V_d$ of every point of $X_d$ is a confined subgroup of $V_d$.

 The following description of the dynamics of individual elements of $V_d$ is due to Brin  \cite[\S 10.7]{Brin-nV}.

 \begin{prop}\label{p-brin}
 For every  $g\in V_d$, there exists a decomposition $X_d=Y_g\sqcup Z_g$ into two $g$-invariant clopen sets such that the following hold:
 \begin{enumerate}[label=(\roman*)]
\item The restriction of $g$ to $Y_g$ has finite order;
\item The subset $\Per_{\operatorname{hyp}}(g)\subset Z_g$ of points with a finite $g$-orbit is finite, and there exists a partition $\Per_{\operatorname{hyp}}(g)=\Att(g)\sqcup \Rep(g)$ such that for every neighbourhood $U$ of $\Att(g)$ and $V$ of $\Rep(g)$, there exists $n>0$ such that $g^{n}(Z_g\setminus V)\subset U$ and $g^{-n}(Z_g\setminus V)\subset U$. 
\end{enumerate}
Moreover $Z_g\neq\varnothing$ if and only if $g$ has infinite order, and in this case $\Att(g)$ and $\Rep(g)$ are both non-empty. 
 \end{prop}

Points in $\Per_{\operatorname{hyp}}(g)$ are called the \textbf{hyperbolic periodic points} of $g$, and $\Att(g)$ and $\Rep(g)$ are respectively the attractive and repelling ones.

Given an action of a group $G$ on a compact space $X$, we say that a closed subset $C\subset X$ is \textbf{compressible} if there exists a point $x\in X$ such that for every neighbourhood $U$ of $x$ there exists $g\in G$ such that $g(C)\subset U$.  Similarly a Borel probability measure $\mu\in \operatorname{Prob}(X)$ is compressible if the closure of $G\cdot \mu$ in the weak-* topology of $\operatorname{Prob}(X)$ contains a Dirac mass $\delta_x$. Recall that the action of $G$ on $X$ is \textbf{proximal} if every pair of points is compressible. When this holds, every finite subset is compressible  \cite[Ch. VI Cor. 1.4]{Margulis}.  The action is \textbf{strongly proximal} if every $\mu\in \operatorname{Prob}(X)$ is compressible. Recall also from Section \ref{sec-appl} that the action is extremely proximal if every proper closed subset $C\subsetneq X$ is compressible. Extreme proximality implies strong proximality, which implies proximality, and the reverse implications are not true in general \cite{Glas-EP}.

In our present setting we deduce from Proposition \ref{p-brin} the following equivalence for subgroups of $V_d$, which was inspired by a reading of \cite[\S 5]{Hu-Mil}.

\begin{prop}\label{p-V-prox}
Let $H\le V_d$ be a subgroup which acts minimally on $X_d$. Suppose that $H$ contains an element of infinite order.  Then the following are equivalent:
\begin{enumerate}[label=(\roman*)]
\item \label{i-prox} the action of $H$ on $X_d$ is proximal;
\item the action of $H$ on $X_d$ is strongly proximal;
\item \label{i-EP} the action of $H$ on $X_d$ is extremely proximal. 
\end{enumerate}
\end{prop}

 \begin{proof}
We only have to prove the implication \ref{i-prox}$\Rightarrow$\ref{i-EP}. First of all note that $\bigcap_{g\in H} Y_{g}$ is a closed $H$-invariant subset of $X_d$. Moreover since $H$ contains an element of infinite order, we have $\bigcap_{g\in H} Y_{g} \subsetneq X_d$. Thus by minimality of the $H$-action we have $\bigcap_{g\in H} Y_{g} =\varnothing$. By the finite intersection property for compact sets, we can find $g_1,\cdots, g_n\in H$ such that $\bigcap_{i=1}^n Y_{g_i}=\varnothing$. Without loss of generality we may assume that all $g_i$'s have infinite order (indeed if $g_i$ has finite order then we have $Y_{g_i}=X_d$ and we can simply remove $g_i$ without changing the conclusion). Moreover, after replacing each $g_i$ by one of its powers, we can assume that $g_i$ is the identity on $Y_{g_i}$.

Now let $C\subsetneq X_d$ be a closed subset of $X_d$, and denote by $\mathcal{O}_C$ the closure of the $H$-orbit of $C$ in the space of all closed subset of $X_d$ (with respect to the Vietoris topology). We aim to show that $\mathcal{O}_C$ contains a singleton. Since finite subsets are compressible by $H$ by the assumption that the $H$-action is proximal, it is enough to show that $\mathcal{O}_C$  contains a subset of the finite set $\Sigma :=\bigcup_{i=1}^n \Per_{\operatorname{hyp}}(g_i)$.   
First observe that by proximality and minimality of the $H$-action, we can find $f\in H$ such that  $f(\Sigma)\subset X_d\setminus C$, so that $C_0:=f^{-1}(C)  \in \mathcal{O}_C$ satisfies $C_0\cap \Sigma=\varnothing$. Now by Proposition \ref{p-brin} we can find a sequence $n_i$ of positive integers tending to $\infty$ such that $g_1^{n_i}(C_0)$ converges in $\mathcal{O}_C$ to a closed subset $C_1 \in  \mathcal{O}_C$ such that $C_1\subset (C_0\cap Y_{g_1})\cup \Per_{\operatorname{hyp}}(g_1)$. Note that since $C_0$ does not intersect $\Rep(g_2)$, every point in $C_1\cap \Rep(g_2)$ must be isolated in $C_1$. So again we can find a sequence $(n_i)$  such that $g_2^{n_i}(C_1)$ converges in $\mathcal{O}_C$  to a closed subset $C_2$ such that $C_2\subset (C_0\cap Y_{g_1}\cap Y_{g_2}) \cup \Per_{\operatorname{hyp}}(g_1) \cup \Per_{\operatorname{hyp}}(g_2)$. Proceeding in this way we can find for each $i=1\cdots, n$ a subset $C_i \in \mathcal{O}_C$ such that $C_i\subset (C_0\cap Y_{g_1}\cap\cdots \cap Y_{g_i})\cup \Sigma$. In particular $C_n\subset \Sigma$ is finite. As observed above, this proves the statement.  \qedhere
 \end{proof}

 \begin{remark}
 The assumption that $H$ has elements of infinite order is necessary. Indeed consider the subgroup $H\le V_d$ consisting of elements defined by partitions $X_d=C_{w_1}\sqcup \cdots \sqcup C_{w_k}=C_{v_i}\sqcup \cdots \sqcup C_{v_k}$ such the words $w_i$ and $v_i$ have the same length for every $i$. Then $H$ is an infinite locally finite group, isomorphic to a block diagonal limit of symmetric groups $\sym(d^n)$, and its action on $X_d$ is minimal, proximal, and  preserves the natural uniform Bernoulli measure on $X_d$. Hence this action is not extremely proximal. We also note that by a result of R\"over \cite{Rover-torsion}, every subgroup of $V_d$ without elements of infinite order is locally finite. 
 \end{remark}

 We will use the following consequence of Proposition \ref{p-V-prox}. Recall that the \textbf{germ} of an element $g\in V_d$ at a point $\xi\in X_d$ is the equivalence class $[g, \xi]$ of the pair $(g, \xi)$ under the equivalence relation defined by $(g_1, \xi_1)\sim (g_2, \xi_2)$ if $\xi_1=\xi_2$ and there exists a neighbourhood $U$ of $\xi_1$ such that $g_1|_U=g_2|_U$.  We say that a subgroup $H\le V_d$ \textbf{covers pairs of germs} in $V_d$ if for every $g\in V_d$ and every pair of points $\xi, \eta\in X_d$ there exists $h\in H$ such that  $[g, \xi] = [h, \xi]$ and $[g, \eta] = [h, \eta]$. Equivalently, $H$ covers pairs of germs if $V_d = (V_d)_{\xi, \eta}^0  H $ for every $\xi, \eta\in X_d$, where $(V_d)_{\xi, \eta}^0$ is the subgroup of $V_d$ consisting of elements acting trivially on an open subset containing $\xi, \eta$.

 From Proposition \ref{p-V-prox} we deduce the following. 
 
 \begin{prop}\label{prop-germ-prox}
Let $H\le V_d$ be a subgroup which covers pairs of germs in $V_d$. Then the action of $H$ on $X_d$ is minimal and extremely proximal.
 \end{prop}

 \begin{proof}
 The assumption implies in particular that $H$ has the same orbits as $V_d$  on pairs of distinct points. Since the action of  $V_d$ on $X_d$ is minimal and proximal, so is the $H$-action. Moreover if $g\in V_d$ is an element of infinite order that admits $\xi\in Z_g$ as a hyperbolic attractive fixed point, then any element $h\in H$ such that $[g, \xi] = [h, \xi]$ also has infinite order. Thus Proposition \ref{p-V-prox} implies that the $H$-action is extremely proximal.  \qedhere
 \end{proof}

 \subsection{Highly transitive actions of $V_d$}

 We now give the proof of the main result of this section.
 
 \begin{thm}\label{t-V-ht}
Every faithful and highly transitive action of $V_d$ on a set is conjugate to its action on an orbit in $X_d$. 

 \end{thm}

In the proof we will invoke the following lemma, which is an immediate consequence of well-known properties of the action of $V_d$ on $X_d$.  Given $G=V_d$ and points $\xi_1, \cdots, \xi_n\in X_d$, we use the notation $G_{\xi_1,\ldots, \xi_n}$ for  the pointwise stabilizer of $\xi_1,\ldots, \xi_n$. 

\begin{lem}\label{l-V-max}
Let $G = V_d$. For every $\xi_1, \xi_2 \in X_d$, the proper subgroups of $G$ that contain $G_{\xi_1, \xi_2}$ are  $G_{\xi_1, \xi_2}, G_{\xi_1}, G_{\xi_2}$, and the stabilizer in $G$ of $\{\xi_1,\xi_2\}$.
\end{lem}

\begin{proof}[Proof of Theorem \ref{t-V-ht}]
In the sequel we denote $G = V_d$. Assume that $G$ acts highly transitively on $\Omega$. Given $\xi \in X_d$, the subgroup $G_\xi$ is a confined subgroup of $G$, and every pair of non-trivial elements with disjoint supports in $X_d$ forms a confining subset for $(G_\xi,G)$. Hence by Theorem \ref{thm-conf-ht-precise} (applied with $r=2$; and note that $G$ is not partially finitary), either $G_\xi$ fixes a unique point $\omega_\xi$ in $\Omega$ and acts highly transitively on the complement, or the action of $G_\xi$  on $\Omega$ is highly transitive.

 Suppose there exists $\xi$ such that the first possibility holds. Then $G_\xi$ being a maximal subgroup of $G$, we must have equality $G_\xi = G_{\omega_\xi}$, and it follows that the $G$-action on $\Omega$ is conjugate to the $G$-action on the orbit of $\xi$. Hence in that case the conclusion holds. In the sequel we assume that for all $\xi$ the action of $G_\xi$  on $\Omega$ is highly transitive, and we want to reach a contradiction. 

We fix a point $\xi \in X_d$. Given $\eta  \in X_d$, the subgroup $G_{\xi,\eta}$ is confined in $G_\xi$, and for the same reason as before there exists a confining subset with two elements. Hence we can apply Theorem \ref{thm-conf-ht-precise} again, but this time to the action of $G_\xi$  on $\Omega$. Suppose that there exists $\eta$ such that $G_{\xi,\eta}$ fixes a point $\omega$ of $\Omega$, i.e. $G_{\xi, \eta}\le G_\omega$. Then by Lemma \ref{l-V-max} we deduce that $G_\omega$ fixes $\xi$, or $G_\omega$ fixes $\eta$, or $G_\omega$ stabilizes  $\{\xi,\eta\}$. If $G_\omega$ fixes $\xi$ then by maximality of $G_\omega$ in $G$ we must have $G_\xi = G_{\omega}$, which is impossible because we make the assumption that the action of $G_\xi$  on $\Omega$ is highly transitive.  By the same argument $G_\omega$ cannot fix $\eta$. And similarly $G_\omega$ cannot stabilize the pair $\{\xi,\eta\}$, because otherwise the action of $G$ on $\Omega$ would be conjugate to the action of $G$ on the orbit of $\{\xi, \eta\}$, which is impossible as the latter is not even 2-transitive. Hence all possibilities lead to a contradiction, and hence there is no $\eta$ such that $G_{\xi,\eta}$ fixes a point in $\Omega$. Therefore by Theorem \ref{thm-conf-ht-precise} we deduce that the action of  $G_{\xi,\eta}$ on $\Omega$ is highly transitive for all $\xi, \eta  \in X_d$.

Since $G_{\xi, \eta}^0$ is normal in $G_{\xi, \eta}$, the action of $G_{\xi, \eta}^0$ on $\Omega$ is also highly transitive. This implies that  whenever $\Sigma$ is a finite subset of $\Omega$, we have $G = G_{\xi, \eta}^0 G_\Sigma$, where $G_\Sigma$ is the pointwise fixator of $\Sigma$; indeed given $g\in G$ there exists $h\in G_{\xi, \eta}^0$ which coincides with $g$ on $\Sigma$, so that $h^{-1}g\in G_\Sigma$. Since the equality $G = G_{\xi, \eta}^0 G_\Sigma$ holds for all distinct points  $\xi,\eta$, we infer that the subgroup $G_\Sigma$  covers pairs of germs in $G$. Therefore we deduce from Proposition \ref{prop-germ-prox} that the action of $G_\Sigma$ on $X_d$ is minimal and extremely proximal.

 In order to derive a contradiction from this point, we adapt an easy argument used in the proof of \cite[Lemma 3.8]{LBMB-trans-degree}. Let $U\subset X_d$ be a non-empty proper clopen subset, and choose an element $k\in \rist_G(U)$ such that $k^2\neq 1$, and a point $\omega_1\in \Omega$ whose $k$-orbit has order $\ge 3$. Set $\omega_2=k(\omega_1), \omega_3=k^2(\omega_1)$. Choose $g\in G$ such that $g(\omega_1)=\omega_1, g(\omega_2)=\omega_2$ and $\omega_3':=g(\omega_3)\neq \omega_3$, and set $\Sigma = \{\omega_1, \omega_2, \omega_3'\}$. Since $G_\Sigma$ acts minimally and extremely proximally, we can choose $h\in G_\Sigma$ such that $hg(U)\cap U=\varnothing$. Set $f=hg$ and $k'=fkf^{-1}$. Note that on the one hand we have $k'(\omega_1)=\omega_2$ and $k'(\omega_2)=\omega_3'$ because $f(\omega_i)=\omega_i$ for $1, 2$ and  $f(\omega_3)=\omega_3'$.  On the other hand $k'$ commutes with $k$ since $k'$ is supported in $f(U)$ and $U$ and $f(U)$ are disjoint. Therefore $k'(\omega_2)=k'k(\omega_1)=kk'(\omega_1)=k(\omega_2)=\omega_3$. So  $\omega_3 = \omega_3'$, and we have reached a contradiction. By the first paragraph above, this terminates the proof.
\end{proof}

\begin{remark} \label{rmq-general-Vd}
The attentive reader will have noticed that the only important point of the proof of Theorem \ref{t-V-ht} where properties of the group $V_d$ are used is Proposition \ref{prop-germ-prox} (which, in turns, relies on Proposition \ref{p-V-prox}). Actually the proof of Theorem \ref{t-V-ht} works for every group of homeomorphisms of a compact space that fulfils the conclusions of Lemma \ref{l-V-max}  and Proposition \ref{prop-germ-prox}. 
\end{remark}
 
 \section{Invariant random subgroups and highly transitive actions} \label{sec-irs}
 
 The purpose of this section is to prove the following. Recall that an IRS of a group $G$ is a $G$-invariant Borel probability measure on $\sub(G)$ .
 
 \begin{prop} \label{prop-IRS-ht}
 	Let $G$ be a countable group that is not partially finitary, and let $\mu$ be an IRS of $G$ such that $\mu(\{1\})=0$. Suppose that $G$ admits a faithful and highly transitive action on a set $\Omega$. Then for $\mu$-almost every $H\in \sub(G)$ the action of $H$ on $\Omega$ is highly transitive.
 \end{prop}
 
 \begin{proof}
We begin by showing that $\mu$-almost every $H$ acts transitively on $\Omega$. To show this, it is enough to show  the following claim: for every  $\omega\in \Omega$ and every probability measure $\lambda$ on $\sub(G)$ that is invariant by $G_\omega$, $\lambda$-almost every $H$ must either fix $\omega$ or act transitively on $\Omega$. Indeed assume that the claim is proven, and consider the events $E_\omega:=\{H \colon H \text{ fixes } \omega \}$ and   $E_t:=\{H\colon H \text{ acts transitively on } \Omega\}$. Clearly $E_\omega\cap E_t=\varnothing$ and the previous claim applied to $\lambda=\mu$ shows that $\mu(E_\omega \sqcup E_t)=1$ for every $\omega$. Since $\Omega$  is countable, this implies that the event $\bigcap_\omega (E_\omega \sqcup E_t)=(\bigcap_\omega E_\omega) \sqcup E_t$ has measure one. However  $\bigcap_\omega E_\omega$ is nothing but the event that $H$ acts trivially on $\Omega$, i.e. that $H$ is the trivial subgroup. Since we are assuming that $\mu(\{1\})=0$, it follows that $\mu(E_t)=1$, as desired.
	
To prove the claim, fix $\omega \in \Omega$ and a probability measure $\lambda$ on $\sub(G)$ that is invariant by $G_\omega$.  By Choquet's theorem \cite[Theorem 27.6]{Choquet},  $\lambda$ can be written as an integral  of $G_{\omega}$-invariant probability measures that are ergodic for the action of $G_{\omega}$.  Thus to prove the claim there is no loss of generaility in assuming that $\lambda$ is $G_{\omega}$-ergodic. We write $\Omega^\ast = \Omega \setminus \left\lbrace \omega \right\rbrace $, and we consider the map $\varphi\colon\sub(G) \to \left\lbrace 0,1  \right\rbrace ^{\Omega^\ast}$ which associates to a subgroup $H$ the $H$-orbit of the point $\omega$, from which we remove $\omega$. The map $\varphi$ is equivariant for the actions of $G_\omega$, and hence the push-forward $\nu$ of $\lambda$ is an invariant ergodic probability measure on $\left\lbrace 0,1  \right\rbrace ^{\Omega^\ast}$. Now since the image of $G_\omega$ inside $\Sym(\Omega^\ast)$ is dense and the action of $\Sym(\Omega^\ast)$ on $\left\lbrace 0,1  \right\rbrace ^{\Omega^\ast}$ is continuous, the measure $\nu$ is also invariant under $\Sym(\Omega^\ast)$. By de Finetti's theorem the ergodic probability measures on $\left\lbrace 0,1  \right\rbrace ^{\Omega^\ast}$ that are invariant under $\Sym(\Omega^\ast)$ are Bernoulli, so it follows that $\nu = (p\delta_0 + (1-p)\delta_1)^{\Omega^\ast}$ for some $p \in [0,1]$. Now given an infinite subset $\Sigma \subset \Omega \times \Omega$ that does not intersect the diagonal, denote by $S_\Sigma$ the set of $(x_\alpha) \in \left\lbrace 0,1  \right\rbrace ^{\Omega^\ast}$ such that $x_\alpha = x_{\alpha'}$ for all $(\alpha,\alpha') \in \Sigma$. Note that if $p \in  ]0,1[$ then $S_\Sigma$ has measure $0$. Now take a non-trivial element $g$ of $G$ such that the set $\H_g$ of subgroups $H$ of $G$ that contain $g$ has positive measure. Since $g$ moves infinitely many points in $\Omega$, one can find a sequence $(\omega_n)$ such that all the points $\omega_n$ and $g(\omega_n)$ are pairwise distinct. If we let $\Sigma$ be the set of pairs $(\omega_n,g(\omega_n))$, then by construction one has $\varphi(\H_g) \subset S_\Sigma$. Hence by the choice of $g$ it follows that $S_\Sigma$ has positive $\nu$-measure, and hence $p=0$ or $p=1$ by the above observation. This means that either $H$ fixes $\omega$ almost surely or $H$ acts transitively almost surely, and the claim is proved. As explained at the beginning of the proof, it follows that $\mu$-almost every $H$ acts transitively on $\Omega$.

 	We shall now argue that almost surely $H$ does not act freely on $\Omega$. Let $S_F$ denote the set of subgroups $H$ of $G$ such that $H$ acts freely on $\Omega$. Suppose for a contradiction that $\mu(S_F) > 0$. Then one can find a non-trivial element $g$ such that $\mu(S_F \cap \H_g) > 0$. Choose $\omega_1,\omega_2$ that are distinct and such that $g(\omega_1) = \omega_2$. Then for every $\gamma \in G$ such that $\gamma(\omega_1) = \omega_2$ and $\gamma \neq g$, then one has that $S_F \cap \H_g \cap \H_\gamma$ is empty. In particular whenever $\gamma$ is a conjugate of $g$ by an element of $G_{\omega_1,\omega_2}$ and $\gamma$ is distinct from $g$, then the sets $S_F \cap \H_g$ and $S_F \cap \H_\gamma$ have the same (positive) measure and do not intersect each other. This implies that the set of conjugates of $g$  by $G_{\omega_1,\omega_2}$ is finite, i.e. $g$ centralizes a finite index subgroup of $G_{\omega_1,\omega_2}$. Since this subgroup acts highly transitively on the complement of $\left\lbrace \omega_1,\omega_2\right\rbrace $ in $\Omega$, this is possible only if $g$ is the transposition that exchanges $\omega_1,\omega_2$. This is absurd since we assume that all non-trivial elements of $G$ have infinite support. So almost surely $H$ does not act freely on $\Omega$.
 	
 	Now given $k \geq 1$, we shall argue by induction that almost surely the action of $H$ on $\Omega$ is $k$-transitive and the fixator in $H$ of $k$ distinct elements of $\Omega$ is non-trivial. We have already treated the case $k=1$. Suppose the result holds for $k$, and choose a subset $\Sigma$ in $\Omega$ of cardinality $k$. The IRS induced on the fixator $G_\Sigma$ by the map $H \mapsto H \cap G_\Sigma$ does not charge the identity by assumption, and hence we can apply the case $k=1$ to the action of $G_\Sigma$  on the complement of $\Sigma$ in $\Omega$. Combined with the induction hypothesis this implies that almost surely the action of $H$ on $\Omega$ is $k+1$-transitive and the fixator in $H$ of $k+1$ distinct elements of $\Omega$ is non-trivial. This terminates the induction, and it immediately follows that almost surely the action of $H$ on $\Omega$ is highly transitive.
 \end{proof}
 
\bibliographystyle{amsalpha}
\bibliography{bib-conf-sym1}

\end{document}